\documentclass[a4paper,oneside,english,12pt,reqno]{amsart}
\usepackage[T1]{fontenc}
\usepackage[hscale=0.68, vscale=0.7]{geometry}
\usepackage[numbers,sort&compress]{natbib}
\usepackage[english]{babel}

\usepackage{amssymb}

\usepackage[np,autolanguage]{numprint}
\usepackage{mathtools}
\DeclarePairedDelimiterX\Ioo[2]{\lparen}{\rparen}{#1,#2}
\DeclarePairedDelimiterX\Iof[2]{\lparen}{\rbrack}{#1,#2}
\DeclarePairedDelimiterX\Ifo[2]{\lbrack}{\rparen}{#1,#2}
\DeclarePairedDelimiterX\Iff[2]{\lbrack}{\rbrack}{#1,#2}
\DeclarePairedDelimiterX\Iffint[2]{\lbrack\!\lbrack}{\rbrack\!\rbrack}{#1,#2}
\newcommand\plusbas{\vphantom{X^X}}
\newcommand\sD[1][]{%
  \if\relax\detokenize{#1}\relax\else{}_{\plusbas#1}\mskip-1mu\relax\fi
  \mathsf{D}%
}
\newcommand\NN{\mathbb{N}}

\newcommand\CC{\mathbb{C}}
\newcommand\dx{\mathrm{d}\mskip-1mu x}

\DeclareMathOperator{\Un}{\mathbf{1}}
\newcommand\dql{\texttt{\textquotedblleft}}
\newcommand\dqr{\texttt{\textquotedblright}}
\allowdisplaybreaks

\DeclareFontFamily{OMX}{mlmex}{}
\DeclareFontShape{OMX}{mlmex}{m}{n}{%
   <->mlmex10%
   }{}%
\usepackage{mlmodern}

    \newtheorem{theorem}{Theorem}

    \newtheorem{proposition}[theorem]{Proposition}
    \newtheorem{remark}[theorem]{Remark}

\usepackage{hyperref}
\hypersetup{%
  colorlinks=true,%
  citecolor=[RGB]{120,29,126},%
  pdfauthor={Jean-François Burnol},%
  pdfsubject={Ellipsephic numbers, Kempner series, Asymptotic expansions},%
  pdfstartview=FitH,%
  pdfpagemode=UseNone,%
}

\title[Kempner-Irwin asymptotics]{On the asymptotics of Kempner-Irwin sums}

\author[J.-F. Burnol]{Jean-François Burnol}

\address{Université de Lille,
  Faculté des Sciences et technologies,
  Département de mathématiques,
  Cité Scientifique,
  F-59655 Villeneuve d'Ascq cedex,
  France}

\email{jean-francois.burnol@univ-lille.fr}

\def \doiurl#1{\url{https://doi.org/#1}}
\date{Now published by the International Journal of Number Theory (2026). See
  \doiurl{10.1142/S1793042126501319}.}


\begin{document}

\begin{abstract}
  Let $I(b,d,k)$ be the subseries of the harmonic series keeping the integers
  having exactly $k$ occurrences of the digit $d$ in base $b$. We prove the
  existence of an asymptotic expansion to all orders in descending powers of $b$,
  for fixed $d$ and $k$, of $I(b,d,k)-b\log(b)$. We explicitly give,
  depending on cases, either four or five terms.  The coefficients involve the
  values of the zeta function at the integers.
\end{abstract}

\maketitle


\section{Ellipsephic harmonic sums and main theorem}

The convergence of the Kempner harmonic sums (1914, \cite{kempner})
$K(b,d)=\sum' m^{-1}$, where denominators are constrained not to contain the
radix-$b$ digit $d$, is the subject of Theorem 144 in Hardy and Wright's work
\cite{hardywright}. This topic is therefore quite widely known.  More
generally, let $I(b,d,k)=\sum^{(k)}m^{-1}$, where the notation means that
$m>0$ has $k$ occurrences of the digit $d$ in base $b$. Here again (Irwin
1916, \cite{irwin}), the proof of convergence is almost immediate. But
obtaining precise numerical values is difficult, and these series generally
converge very slowly. It seems that we had to wait until 1979 and the work of
Baillie \cite{baillie1979} for the $K(10,d)$ for $d=0,1,\dots,9$ to be
tabulated with $20$ decimal places each, whereas Kempner had only
indicated $K(10,9)<90$ and Irwin had obtained
$\np{22.4}<K(10,9)<\np{23.3}$. See \cite{baillie2008} for Baillie
computational algorithms for $I(b,d,k)$, $k>0$.

Since the early work of Kempner and Irwin, this topic has continued to
arouse interest and has been the subject of various studies and generalizations
(estimations, infinite products, Dirichlet series, combinatorics, etc.).
Here are some references for the interested reader:
\cite{allouchecohenetalActa1987, alloucheshallit1989,
  alloucheshallitCUP2003, alloucheshallitsondow2007,
  allouchemorin2023, allouchehumorin2024, aloui2017, baillie1979,
  baillie2008, burnolblocks,
  farhi, fischer, gordon2019, huyining2016, irwin, kempner,
  klove1971, kohlerspilker2009, lubeckponomarenko2018, mukherjeesarkar2021,
  nathanson2021integers, nathanson2022ramanujan, schmelzerbaillie,
  segalleppfine1970, wadhwa1978, walker2020}.

The sequence of integers lacking certain digits in a base $b$ has been the
subject of advanced work in analytic number theory. We cite here only a
few articles from a vast body of literature: \cite{erdosmauduitsarkozy1998},
\cite{dartygemauduit2000}, \cite{maynardInventiones2019}.
In this context, Ch.\@~Mauduit introduced the neologism ``ellipsephic'' to refer
to integers that lack certain digits in a base $b$. Details on the etymology
can be found in \cite{allouchemorin2023}. We could therefore speak of
``ellipsephic harmonic sums'' for $K(b,d)$ and, by extension, also when
referring to
$I(b,d,k)$, $k>0$, even if the digit $d$ is then not completely missing.

The main focus of this paper is the asymptotic expansion of $I(b,d,k)$ for
$b\to\infty$. Here next is a combined (shortened) version of the main
results:
\begin{theorem}\label{thm:main}
  For fixed $d$ and $k$, $I(b,d,k)-b\log(b)$ has an asymptotic expansion
  to all orders in inverse powers of $b$.  The first occurring power of $b$
  is $b^{-2k-1}$
  if $d=0$ and $b^{-2k+1}$ if $d>0$.
  For $d>0$ and $k=0$, the term proportional to $b$ is
  $-b\log(1+\frac1d)$.

In detail, with the implicit constants in $O(b^{-n})$ depending on both $d$ and $k$:
\begin{align*} 
&\mathrel{\phantom{\implies}}I(b,0,0) = b\log(b)+\frac{\zeta(2)}{2b}-\frac{\zeta(3)}{3b^2} + O(b^{-3}),
\\ 
k>0&\implies I(b,0,k)=b\log(b)+\frac{\zeta(2)}{2b^{2k+1}}+\frac{k\zeta(2)}{2b^{2k+2}}+O(b^{-2k-3}),
\\ 
d>0&\implies I(b,d,0) = b\log(b) 
\begin{aligned}[t] 
&-b \log(1+\frac1d)+\frac{(d+\frac12)(\zeta(2)-d^{-2})}{b} 
\\ 
&-\frac{(d^2+d+\frac13)\zeta(3)-\frac{9d^2+8d+2}{6d^3(d+1)}}{b^2}+O(b^{-3}),
\end{aligned}
\\ 
d>0&\implies I(b,d,1) = b\log(b) + \frac{d+\frac12}{d^2b}-\frac{9d^2+8d+2}{6d^3(d+1)b^2}+O(b^{-3}),
\\
\begin{pmatrix} 
d>0\\k>1
\end{pmatrix}
&\implies I(b,d,k) = b\log(b)
\begin{aligned}[t]
  &+ \frac{d+\frac12}{d^2b^{2k-1}}\\
  &+\frac{(d+\frac12)\bigl((k-2)d+k-3\bigr)}{d^2(d+1)b^{2k}} + O(b^{-2k-1}).
\end{aligned}
\end{align*}
\end{theorem}
In the following text, we shall give up to four or five terms in all cases:
Prop.\@ \ref{prop:develKb0} for $d=0$ and $k=0$ (five terms), Prop.\@
\ref{prop:develIb0k} for $d=0$ and $k>0$ (four terms), Prop.\@
\ref{prop:develKbd} for $d>0$ and $k=0$ (five terms), Prop.\@
\ref{prop:develIbd1} for $d>0$ and $k=1$ (five terms), Prop.\@
\ref{prop:develIbd2} for $d>0$ and $k=2$ (five terms), Prop.\@
\ref{prop:develIbdk} for $d>0$ and $k\geq3$ (four terms). If we wanted a fifth
term in this last case, we would have to distinguish between $k=3$, $k=4$, and
$k\geq5$.

The values of the Riemann zeta function at positive integers appear
in all these expansions, even if, as seen above, the first two
terms may not involve them.  This happened already in the author
earlier work \cite{burnollargeb}, which established
\begin{equation*}
  K(b,b-1) = b\log(b) - \frac{\zeta(2)}{2b} -
  \frac{3\zeta(2)+\zeta(3)}{3b^2} - \frac{2\zeta(2)+4\zeta(3)+\zeta(4)}{4b^3}
  + O(b^{-4}).
\end{equation*}
In the present paper, the digit $d$ is fixed independently of the base
$b$. The strategy, due to the increased difficulty introduced by $k$,
differs from that of \cite{burnollargeb}: firstly, instead of starting from
the geometrically convergent series from 
\cite{burnolkempner,burnolirwin}, which use the moments $u_{k;m}$ of certain
measures $\mu_k$ on $\Ifo{0}{1}$, we begin by expressing the difference
with $b\log(b)$ as an integral against these very measures, and involving the digamma
function. This generalizes a result from \cite{burnoldigamma}, for the situation considered
here of a single excluded
digit. Secondly, we do not seek uniformity
in $m$, when approximating the values of the moments $u_{k;m}$.

Theorem \ref{thm:main} lacks uniformity in $k$ and therefore, in itself, does
not imply the theorem
\cite{farhi,segalleppfine1970,burnolirwin} on the limit
for $k\to\infty$. But the formulas,
such as the one from Proposition \ref{prop:Ibdk}, from which the
asymptotics presented in this work are derived, yield this $k\to\infty$
limit directly, once it is known from \cite{burnolirwin} that the measures
$\mu_k$ converge weakly as $k\to\infty$ to $b\dx$.

The first two sections are dedicated to the study of the moments $u_{k;m}$ as
analytic (rational) functions of the variable $c=b^{-1}$. This is done
mainly at the formal level for $c\to0$ since we are not attempting to provide
uniform estimates in $m$.
In the third section, we express $I(b,d,k)-b\log(b)$ as a combination of
integrals, against the measures $\mu_k$ (and $\mu_{k-1}$), of the digamma
function, thus generalizing \cite{burnoldigamma} which had dealt with
$k=0$. The use of this function (and the resulting appearance of the
values of the Riemann zeta function at positive integers)
has a precedent in the work of Fischer \cite{fischer}.  It is needed
to treat separately $d=0$ (which turns out to be close to the $d=b-1$ case
examined in
\cite{burnoldigamma,burnollargeb}) from  $d>0$.
The four sections that follow deal successively with $d=k=0$, $d=0$
and $k>0$, $d>0$ and $k=0$, $d>0$ and $k>0$. In all cases, we transform the
integrals into series with geometric convergence, which could be used
numerically, giving $I(b,d,k)-b\log(b)$. These convergent series (for
example, for $d>0$ and $k>0$ see Proposition \ref{prop:Ibdk}) involve the moments
of the measures defined in \cite{burnolkempner,burnolirwin}. We transform them into
asymptotic series in inverse powers of $b$. Some contributions are analyzed
via Euler-Maclaurin summations.
The coefficients in the expansion become, when the exponent in $b^{-1}$
increases, quickly rather large expressions in $k$ and $d$ (and the zeta
values).
For a given exponent, the dependency on $k$
stabilizes to a certain polynomial, but this happens for larger and larger
$k$'s when the exponent increases.
Depending on the case, we give between four and five terms.  For this, it
proved enough to obtain, for the Laurent series expansion at $c=0$ of the
moments $u_{k;m}$, only two terms, exceptionally three, beyond the simple poles.

The obtained theoretical results are then compared with numerical data over a
range of $b$'s. To our knowledge, there are only three known ways to
accurately calculate $I(b,d,k)$ numerically: either using Baillie's code
\cite{baillie2008}, or the one of the author \cite{burnolirwin} (see
\url{https://burnolmath.gitlab.io/irwin}), or the, not yet implemented, code
that would use the series mentioned in the previous paragraph to directly give
the deviation from $b\log(b)$. We used here the second option, i.e.\@ the
\textsf{SageMath} code from \url{https://gitlab.com/burnolmath/irwin}, to
check how $I(b,d,k)-b\log(b)$ behaves as $b$ increases by a factor $2$, $5$ or
$10$. This numerical algorithm
also uses the moments $u_{k;m}$, but in a purely numerical way, and in linear
proportion for the $m$-range to the number of digits requested; whereas our
work here primarily uses the formal properties of $u_{k;m}$ as analytic
functions of $b$, and requires, for the given terms, only $u_{j;0}$,
$u_{j;1}$, \dots, up to at most $u_{j;4}$, and for $j\in\{k,k-1,k-2\}$. The
observed
numerics provide thus a mutual
validation of the formulas from \cite{burnolkempner,burnolirwin} on which the
numerical code was based, and of the asymptotic expansions proven in the
present paper, which is based on different principles.

\section{Measures and moments associated with Irwin sums}

In this section, we review the tools of
\cite{burnolkempner,burnolirwin,burnoldigamma} and some results
necessary here.
For an integer $b>1$, we denote by $\sD = \{0,\dots,b-1\}$ the set of digits
for the base $b$. The union of the Cartesian products of the $\sD^l$ for
$l\geq0$ is called the string space and is denoted by $\boldsymbol{\sD}$. The
length $l\geq0$ of a string $X$ is the integer such that $X\in\sD^l$, and we
denote it by $l(X)$ or $|X|$. A string $X=(d_l,\dots,d_1)$ is also written
$X=\dql d_l\dots d_1\dqr$ or simply $X = d_l\dots d_1$ without the quotation
marks and defines an integer $n(X) = d_{l} b^{l-1} + \dots + d_1b^{0}$ which
belongs to $\Iffint{0}{b^l-1}$. We also associate with $X$ the real number
$x(X) = n(X)/b^{|X|}\in \Ifo{0}{1}$. When $X$ is the empty string, $x(X)=0$.

Let $0\leq d<b$ and $k\geq0$ be fixed integers. We define a measure $\mu_k$ on
the space $\boldsymbol{\sD}$ (equipped with the sigma-algebra of all subsets)
of strings by assigning zero mass to any string that does not have
\emph{exactly} $k$ occurrences of the integer $d$, and mass $b^{-|X|}$ to each
$X$ that has \emph{exactly} $k$ occurrences of $d$. For example, if $k=0$,
$b=10$, $d=9$, the empty string receives a mass of $1$, the string \texttt{99}
receives zero mass, and the string \texttt{001122} receives a mass of
$10^{-6}$. If now, still with $b=10$ and $d=9$, we have $k=1$, the empty
string receives zero mass, the string \texttt{99} again has zero mass, and
\texttt{001122} has zero mass while \texttt{0011922} will have mass $10^{-7}$.

It turns out (\cite{burnolirwin}) that $\mu_k$ confers to
the space $\boldsymbol{\sD}$ of all strings a finite total mass which
is equal to $b$, regardless of the values of $k$ and $d$.

By abuse of notation, we also denote by $\mu_k$ the measure on $\Ifo{0}{1}$
(equipped with the Borel sigma-algebra) obtained by push-forward under
$X\mapsto x(X)$. The measure $\mu_k$ also depends, of course, on $b$ and $d$;
the notation is simplified for readability.

The $\mu_k$ on $\Ifo{0}{1}$ converge weakly to $b$ times the Lebesgue measure
on $\Ifo{0}{1}$. Two separate proofs in \cite{burnolirwin} have established,
firstly, the convergence of the moments $u_{k;m} = \mu_k(x^m)$ of these
measures to those of $b\dx$, and secondly, the convergence for any interval
$I$ of $\mu_k(I)$ to $b$ times the length $|I|$ of $I$.
An extension to counting the occurrences of a multi-digit subblock has been
proven in \cite{burnolblocks}.

It has also been shown in \cite{burnolirwin} that the $u_{k;m}$ form an increasing
sequence for each fixed $m$ (thus with limit $b/(m+1)$ for $k\to\infty$) and
that the complementary moments $v_{k;m} = \mu_k((1-x)^m)$ form, for a fixed
$m$, a decreasing sequence in $k$ (also with limit $b/(m+1)$).

The Kempner-Irwin harmonic sum:
\begin{equation}
I(b,d,k) =\sum\nolimits^{(k)}\frac1m\;,
\end{equation}
where the exponent $(k)$ means that only integers $m>0$ with $k$ occurrences
of the digit $d$ in base $b$ are retained, can be expressed in several
ways starting from the measure $\mu_k$, the simplest expression being:
\begin{equation}\label{eq:Ikmu}
I(b,d,k) =
\int_{\Ifo{b^{-1}}{1}}\frac{\mu_k(dx)}{x}\;.
\end{equation}
The theorem $\lim_{k\to\infty} I(b,d,k)=b\log(b)$
(\cite{farhi,segalleppfine1970}) can therefore be seen via Eq.\@ \eqref{eq:Ikmu} as
a consequence of the convergence theorem from $\mu_k$ to $b\dx$.

In
\cite{burnolkempner} and \cite{burnolirwin}, the functions
\begin{equation}\label{eq:Uk}
U_k(z) = \int_{\Ifo{0}{1}} \frac{\mu_k(\dx)}{z+x}
\end{equation}
play an important role via the equations we now recall. By the definition of
the measure $\mu_k$, $U_k(z)$ is, for all $z\in\CC\setminus\Iff{-1}{0}$, equal
to $\sum m^{-1}$ where only those \emph{complex numbers} $m$ (which are
\emph{integers} if $z$ is an integer) appear, which are of the shape $m = b^l
z + n(X)$, $|X| = l$, $k(X) = k$, $l\geq0$.  With $z$ being an integer $n$, $n>0$, they are
the integers which ``start with $n$'' and whose ``tail'' has exactly $k$ times
the digit $d$.  It follows from the interpretation as a series (which, for an
integer $z=n>0$, is therefore a subseries of the harmonic series, but not
necessarily of the series defining $I(b,d,k)$), that for all
$z\in\CC\setminus\Iff{-1}{0}$ we have:
\begin{equation}\label{eq:U0rec}
U_0(z) = \frac1z + \sum_{a\in\sD,a\neq d} U_0(bz + a),
\end{equation}
and for $k\geq1$:
\begin{equation}\label{eq:Ukrec}
U_k(z) = U_{k-1}(bz+d) + \sum_{a\in\sD,a\neq d} U_k(bz + a).
\end{equation}
Note that indeed $bz+a\notin\Iff{-1}{0}$ for all $a\in\sD$
and $z\in\CC\setminus\Iff{-1}{0}$.

We can integrate against
$\mu_k$ any bounded function $g$ (without continuity conditions) and
we have the following equations \cite[\S5, Lem.\@ 1]{burnolirwin}:
\begin{equation}\label{eq:lemint0}
\int_{[0,1)} g(x)\,d\mu_0(x) = g(0)
+ \frac1b\sum_{a\neq d}\int_{[0,1)} g(\frac{a+x}b)\,d\mu_0(x),
\end{equation}
and, for $k\geq1$: 
\begin{equation}\label{eq:lemint1} 
\int_{[0,1)} g(x)\,d\mu_k(x) = 
\frac1b\sum_{\substack{0\leq a < b\\a\neq d}}\int_{[0,1)} g(\frac{a+x}b)\,d\mu_k(x) 
+ \frac1b \int_{[0,1)} g(\frac{d+x}b)\,d\mu_{k-1}(x).
\end{equation}
Relationships Eqs.\@ \eqref{eq:U0rec} and \eqref{eq:Ukrec} are special cases of this.

\section{Moments as analytic functions of the base}

The moments $u_{k;m}=\int_{\Ifo{0}{1}}x^m\mu_k(\dx)$ form, for each fixed $k$,
a decreasing sequence tending to zero. Except for $b=2$, $d=1$, and
$k=0$, in which case $u_{0;m}=0$ for all $m\geq1$, they are strictly positive
and form a strictly decreasing sequence for $m\to\infty$. We also know
\cite[\S5, Prop.\@ 7]{burnolirwin} that for each fixed $m$, the sequence $k\mapsto u_{k;m}$ is
strictly increasing and has the limit $b/(m+1)$.

The $u_{k;m}$ satisfy recurrences in $k$ and $m$ established in
\cite[\S5, Eq.\@ (4) \& (5)]{burnolirwin}. To state them, let us first define some
coefficients $\gamma_j$:
\begin{equation}
\gamma_j = \sum_{\substack{0\leq a<b\\a\neq d}} a^j\,.
\end{equation}
Again, the notation is very abbreviated because $\gamma_j$ depends on the integers
$b$ and $d$. In particular, $\gamma_0$ equals $b-1$.

It holds for
$k=0$ and $m\geq1$:
\begin{equation}\label{eq:recurru0}
(b^{m+1} - b +1)u_{0;m} = \sum_{j=1}^m \binom{m}{j}\gamma_ju_{0;m-j}\,,
\end{equation}
and $u_{0;0} = b$.
For $k\geq1$ and all $m\geq0$, there holds:
\begin{equation}\label{eq:recurruk}
(b^{m+1} - b +1) u_{k;m}=\sum_{j=1}^{m}\binom{m}{j}\gamma_ju_{k;m-j}
+ \sum_{j=0}^{m}\binom{m}{j}d^ju_{k-1;m-j}\,.
\end{equation}
Given the value of $\gamma_0$, this can also be written
equivalently as:
\begin{equation}\label{eq:recurrukbis}
b^{m+1} u_{k;m}=\sum_{j=0}^{m}\binom{m}{j}\gamma_ju_{k;m-j}
+ \sum_{j=0}^{m}\binom{m}{j}d^ju_{k-1;m-j}\,,
\end{equation}
and we can perform the analogous transformation as in Eq.\@ \eqref{eq:recurru0}.

We have $u_{k;0} = b$ for all $k$.

For $m=1$, the recurrence relations are written:
\begin{align}
u_{0;1} &= \frac b2 -(d+\frac12)\frac b{b^2 - b +1}\;,\\
u_{k;1} &= \frac{b(b^2-b)}{2(b^2 - b +1)} + \frac{u_{k-1;1}}{b^2 - b + 1}\,.
\end{align}
We deduce an explicit formula for $u_{k;1}$:
\begin{proposition}
For all $k\geq0$:
\begin{equation}
u_{k;1} = \frac b2 - \frac{b(2d+1)}{2(b^2-b+1)^{k+1}}\,.
\end{equation}
\end{proposition}
\begin{proof}
A simple check that the proposed formula is compatible with the recurrence relation.
\end{proof}

We will now obtain results relating to $u_{k;m}$ viewed as
analytic functions of $b$, or of $c=b^{-1}$.
Let us define, for all integers $k$ and $m$:
\begin{equation}
w_{k;m} = \frac{b}{m+1}-u_{k;m}\,.
\end{equation}
With these new quantities, Eq.\@ \eqref{eq:recurrukbis}, which
applies for $k\geq1$, is:
\begin{equation}\label{eq:recurrukter}
  \begin{split}
    b^{m+1} \Bigl(-w_{k;m} + \frac{b}{m+1}\Bigr)={}
    &\sum_{j=0}^{m}\binom{m}{j}\gamma_j\Bigl(-w_{k;m-j}+\frac{b}{m+1-j}\Bigr)
    \\
    &+ \sum_{j=0}^{m}\binom{m}{j}d^j \Bigl(-w_{k-1;m-j}+\frac{b}{m+1-j}\Bigr).
  \end{split}
\end{equation}
Terms other than those involving $w_{k;m-j}$
and $w_{k-1;m-j}$ cancel each other out. Indeed
\begin{equation}\label{eq:foo} 
  \begin{split}
    \forall x\quad \sum_{j=0}^m \binom{m}{j}\frac{b x^j}{m-j+1} &= \sum_{j=0}^m \frac b{m+1}\binom{m+1}{j} x^j
    \\&= \frac{b}{m+1}\bigl((x+1)^{m+1} - x^{m+1} \bigr),
  \end{split}
\end{equation}
and therefore
\begin{equation}\label{eq:sommebinom} 
\sum_{0\leq a<b}\sum_{j=0}^m \binom{m}{j}a^j\frac{b}{m-j+1} = \frac{b}{m+1}b^{m+1}.
\end{equation}
Noting now that $\gamma_j + d^j = \sum_{0\leq a<b} a^j$, we see that
Eq.\@ \eqref{eq:sommebinom} combined with Eq.\@ \eqref{eq:recurrukter} produces for the
$w_{k;m}$ the same recurrence as for the $u_{k;m}$:
\begin{equation}\label{eq:recurwkm}
(b^{m+1} - b +1) w_{k;m}=\sum_{j=1}^{m}\binom{m}{j}\gamma_jw_{k;m-j}
+ \sum_{j=0}^{m}\binom{m}{j}d^jw_{k-1;m-j}\,.
\end{equation}
(We used $\gamma_0=b-1$ to move $w_{k;m}$ to the left-hand side only). But
the sums actually stop at $j=m-1$ since $w_{k;0} = 0$ for all $k$.

The preceding was for $k\geq1$. For $k=0$, the calculation is slightly
different (but again uses Eqs.\@ \eqref{eq:foo} and \eqref{eq:sommebinom}):
\begin{align}\notag
b^{m+1} \Bigl(-w_{0;m} + \frac{b}{m+1}\Bigr)
={} &\sum_{j=0}^m \binom{m}{j}\gamma_j\Bigl(-w_{0;m-j}+\frac {b}{m+1-j}\Bigr)
\\
\begin{split}
  ={} &- \sum_{j=0}^m \binom{m}{j}\gamma_jw_{0;m-j}
  \\&+\frac b{m+1}\Bigl(b^{m+1}-(d+1)^{m+1}+d^{m+1}\Bigr),
\end{split}
\\
\begin{split}
  (b^{m+1} - b +1)w_{0;m}
  ={} &\sum_{j=1}^m
  \binom{m}{j}\gamma_jw_{0;m-j}\\&+\frac b{m+1}\Bigl((d+1)^{m+1}-d^{m+1}\Bigr).
\end{split}
\end{align}

Now let $c=b^{-1}$. As this paper studies the limit $b\to\infty$,
we consider $c$ as a quantity tending to zero.

Recall that $c^{j+1}\sum_{0\leq a<b}a^j$ is a polynomial in the variable $c$
which can be expressed using Bernoulli numbers ($B_0=1$, $B_1=-\frac12$,
$B_2=\frac16$, $B_3=0$, $B_4=-\frac1{30}$, ...)  by the formula (the last
equality assuming $j\geq2$):
\begin{equation}
c^{j+1}\sum_{0\leq a<b}a^j = \sum_{p=0}^j \binom{j}{p}B_p\frac{c^p}{j+1-p}
= \frac1{j+1} - \frac c2 + \frac{jc^2}{12}+\dots\;.
\end{equation}
The degree therefore depends on the parity of $j$. The quantities $c^{j+1}\gamma_j$
are also polynomials in $c$, assuming here that $d$ is fixed to a value
independent of $b$:
\begin{align}\label{eq:bar}
c^{j+1}\gamma_j &= \sum_{p=0}^j \binom{j}{p}B_p\frac{c^p}{j+1-p} - d^j c^{j+1}
\\
&=
\begin{cases}
1 - c & (j=0),
\\
\frac12 - \frac12 c - d c^2 & (j=1),
\\
\frac1{j+1} - \frac12 c + \frac{j}{12} c^2 + O^{(j)}_{c\to0}(c^3)&(j\geq2).
\end{cases}
\end{align}
The $O^{(j)}_{c\to0}(c^3)$ simply groups together in the polynomial expression
Eq.\@ \eqref{eq:bar} for $c^{j+1}\gamma_j$ all terms of degree at least $3$ in
$c$. There is no uniformity in $j$ and we simply highlight the coefficients of
$c^0$, $c^1$, and $c^2$.

From the previous calculations, for $m\geq1$:
\begin{equation}\label{eq:recurw}
(1 - c^m + c^{m+1})w_{0;m} = \sum_{j=1}^m \binom{m}{j}c^{m-j}c^{j+1}\gamma_jw_{0;m-j}
+\Bigl((d+1)^{m+1}-d^{m+1}\Bigr)\frac{c^m}{m+1}\;.
\end{equation}
We have
$w_{0;0} = 0$ and (as seen previously)
\begin{equation}
w_{0;1} = (d+\frac12)\frac{c}{1 - c + c^2}\;.
\end{equation}
In the recurrence relation Eq.\@ \eqref{eq:recurw} the summation actually takes
place over $1\leq j\leq m-1$ since $w_{0;0}=0$. As in \cite[\S3, proof of
Prop.\@ 3.1]{burnollargeb}, it
follows that the $w_{0;m}$ are rational fractions of $c$ whose poles lie
outside the open disk $D(0,\rho)$ with $\rho^2(\rho+1) = 1,
\rho\approx\np{0.755}$.

In the summation over $j$, the terms contributed by $j\leq m-2$ are
$O(c^3)$. Therefore, up to $O(c^3)$, it suffices to look at the contribution
of $j=m-1$, which only matters to $c^2$. The term
outside the sum is $(2d+1)c^2/3$ for $m=2$ and $O(c^3)$ for $m=3$ and $O(c^4)$
for $m\geq4$. We can therefore write $w_{0;m} = \alpha_{0;m} c^2 + O(c^3)$ and
we obtain
\begin{equation}
(1 - c^m + c^{m+1})\alpha_{0;m} c^2 =
m c \frac1{m} (d+\frac12) c + \delta_{m=2}(3d^2+3d+1)c^2/3 + O(c^3).
\end{equation}
Therefore $\alpha_{0;2} = d^2+2d+\frac{5}{6}$ while $\alpha_{0;m} = d + \frac12$ for
$m\geq3$. We can then obtain the $c^3$ term, leaving the proof using
mathematical induction to the reader:
\begin{align}
  \frac b2-u_{0;1}=w_{0;1} &= (d+\frac12)\frac{c}{1 - c + c^2},\\
  \frac b3-u_{0;2}=w_{0;2} &= (d^2+2d+\frac{5}{6})c^2+O(c^4),\\
  \frac b4-u_{0;3}=w_{0;3} &= (d+\frac12)c^2+(d+\frac12)d(d+1)c^3+O(c^4),\\
\notag
  \frac b{m+1}-u_{0;m}=w_{0;m} &=(d+\frac12)c^2-(d+\frac12)(\frac m2 - 1)c^3\\
                              &\phantom{={}}+O_m(c^4)\quad(m\geq4).
\end{align}

We prove similar results for
$w_{k;m}=b/(m+1)-u_{k;m}$, $k\geq1$.
\begin{proposition}\label{prop:develu}
  Let $k\geq0$. The $w_{k;m} = \frac{b}{m+1}-u_{k;m}$ are rational fractions
  in $c=b^{-1}$ whose poles lie outside the open disk of radius $\rho$ with
  $\rho^2(\rho+1)=1$, $\rho\approx\np{0.755}$. The order of the multiplicity
  of $c$ as zero of $w_{k;m}$ is $2k+2$ when $m\geq2$.
  We have $w_{k;1} = (d+\frac12)\frac{c^{2k+1}}{(1-c+c^2)^{k+1}}$.  For
  $k\geq2$:
\begin{align} 
w_{k;1} &= (d+\frac12)c^{2k+1} + (d+\frac12)(k+1) c^{2k+2} + O_k(c^{2k+3}),\\ 
w_{k;2} &= 2(d+\frac12)^2c^{2k+2} + (d+\frac12)^2(2k+2)c^{2k+3}+O_k(c^{2k+4}),\\ 
w_{k;3} &= (d+\frac12) c^{2k+2} + (d+\frac12)(3d^2+\frac{2k-1}2)c^{2k+3}+O_k(c^{2k+4}),\\
m\geq4\implies w_{k;m} &= (d + 1/2)c^{2k+2} - (d + 1/2)(m^2 - k-1)c^{2k+3} + O_{k,m}(c^{2k+4}).
\end{align}
For $k\in\{0,1\}$ the same formulas apply with the exception of
\begin{align} 
w_{0;2}&= (d^2+2d+5/6)c^2+0\cdot c^3+O(c^4),\\
w_{0;3}&= (d + 1/2)c^2+(d + 1/2)d(d+1)c^3+O(c^4),\\
w_{1;2}&= 2(d + 1/2)^2c^4 + (3d^2 + 4d + 4/3)c^5 + O(c^6).
\end{align}
\end{proposition}
\begin{proof}
  We already know the $w_{k;1} = (d+\frac12)\frac{c^{2k+1}}{(1-c+c^2)^{k+1}}$
  and therefore the start of their expansion in powers of $c$. So, for
  $k=0$ as well as for all $k$ and $m=1$, the statements made about the
  expansions in power series of $c$ have already been established.

  We can write Eq.\@ \eqref{eq:recurwkm} with the variable $c=b^{-1}$:
\begin{equation}\label{eq:recurwkmc}
  \begin{split}
    (1 - c^m + c^{m+1}) w_{k;m}={}&\sum_{j=1}^{m-1}\binom{m}{j}c^{m-j}c^{j+1}\gamma_jw_{k;m-j}
    \\&+ \sum_{j=0}^{m-1}\binom{m}{j}c^{m+1}d^jw_{k-1;m-j}\,.
  \end{split}
\end{equation}
It follows that these are rational fractions whose poles are among the roots
of $(1-c+c^2)(1-c^2+c^3)\dots(1-c^m+c^{m+1})$, and they have moduli at
least equal to $\rho$ (see \cite{burnollargeb} for details).

First, consider $w_{1;2}$. We have:
\begin{equation}\label{eq:recurw12c}
(1 - c^2 + c^3)w_{1;2} = 2c^3\gamma_1 w_{1;1} + c^3(w_{0;2} + 2 d w_{0;1}),
\end{equation}
with $c^2\gamma_1 = \frac12 - \frac12 c - dc^2$. Thus:
  \begin{equation}
    \begin{split}
      (1+O(c^2))w_{1;2} = {} &c(1 - c)(d+\frac12)(c^3+2c^4)
      \\&+c^3\biggl((d^2+2d+\frac{5}{6})c^2 +2d (d+\frac12)(c+c^2)\biggr) + O(c^6),
    \end{split}
  \end{equation}
which, after simplification, gives the formula from the proposition.

Now we look at $w_{1;3}$. 
\begin{equation} 
(1 - c^3 + c^4)w_{1;3} = 3c^2 c^2\gamma_1 w_{1;2} + 3c c^3\gamma_2 w_{1;1} + c^4(w_{0;3}+3dw_{0;2}+3d^2w_{0;1}).
\end{equation}
The contribution $3c^2 c^2\gamma_1 w_{1;2}$ is $O(c^6)$. The second term gives:
\begin{equation}
  3c c^3\gamma_2 w_{1;1} = 3c(\frac13 - \frac12 c)(d+\frac12)(c^3+2c^4) + O(c^6)= c^4(1+\frac c2)(d+\frac12) + O(c^6).
\end{equation}
The last term gives $c^4$ times $O(c^2)+O(c^2)+3d^2(d+\frac12)c$. In total, we find:
\begin{equation}
w_{1;3} = (d+\frac12)c^4 + \biggl(\frac12(d+\frac12)+3d^2(d+\frac12)\biggr) c^5 + O(c^6),
\end{equation}
which gives the formula from the proposition.

Now consider $w_{1;m}$ for $m>3$. In the second summation, the $w_{0;m-j}$ are
of at least order $c^2$ for $j\leq m-2$, and there is a global factor
$c^{m+1}$, so we can neglect them. The term in $j=m-1$ gives $m
c^{m+1}d^{m-1}w_{0;1}$, or $w_{0;1} = O(c)$, so we obtain $O(c^6)$ for all
$m\geq4$, and therefore it is also negligible. In the first sum of the
recurrence formula Eq.\@ \eqref{eq:recurwkmc}, the terms with $j\leq m-2$ have a
coefficient $c^{m-j}$ of order at least $c^2$, which multiplies $w_{1;m-j}$,
which is $O(c^4)$. Therefore, up to $O(c^6)$, we can neglect them. The term in
$j=m-1$ is $mc(\frac1m - \frac12c)w_{1;1}$, so it contributes $c(1 -
\frac{m}{2}c)(d+\frac12)c^3(1
+2c)=(d+\frac12)(c^4+(2-\frac{m}{2})c^5)+O(c^6)$, and we obtain the formula
from the proposition.

Now we look at $w_{2;m}$. We already have the explicit formula for
$w_{2;1}$.
\begin{equation}
(1 - c^2 + c^3)w_{2;2} = 2c^3\gamma_1 w_{2;1} + c^3(w_{1;2} + 2 d w_{1;1}),
\end{equation}
with $c^2\gamma_1 = \frac12 - \frac12 c + O(c^2)$. Thus:
  \begin{equation}
    \begin{split}
      (1+O(c^2))w_{2;2} = {}& c(1 - c)(d+\frac12)(c^5+3c^6)
      \\&+c^3\biggl(2(d+\frac12)^2 c^4 +2d (d+\frac12)(c^3+2c^4)\biggr) + O(c^8).
    \end{split}
  \end{equation}
and after simplification $ w_{2;2} = 2(d+\frac12)^2 c^6 + 6(d+\frac12)^2 c^7 +
O(c^8)$ which is the formula from the proposition.

The formula for $w_{2;3}$ is:
\begin{equation}
  (1 - c^3 + c^4)w_{2;3} = 3c^2 c^2\gamma_1 w_{2;2} + 3c c^3\gamma_2 w_{2;1} + c^4(w_{1;3}+3dw_{1;2}+3d^2w_{1;1}).
\end{equation}
The contribution $3c^2 c^2\gamma_1 w_{2;2}$ is $O(c^8)$. The second term
gives:
\begin{equation}
3c c^3\gamma_2 w_{2;1} = 3c(\frac13 - 
   \frac12 c)(d+\frac12)(c^5+3c^6) + O(c^8)= c^6(1+\frac 32 c)(d+\frac12) + O(c^8).
\end{equation}
The last term gives $c^4$ times $O(c^4)+O(c^4)+3d^2(d+\frac12)c^3$. In total, we find:
\begin{equation}
w_{2;3} = (d+\frac12)c^6 + \biggl(\frac32(d+\frac12)+3d^2(d+\frac12)\biggr) c^7 + O(c^8).
\end{equation}
which gives the formula from the proposition.

We then calculate $w_{2;m}$ for $m\geq4$ as we did for $w_{1;m}$. Then, by
induction on $k$, we repeat the previous arguments, first calculating
$w_{k;2}$ and $w_{k;3}$, which require a separate discussion, before
dealing with $w_{k;m}$ for $m\geq4$.
\end{proof}

\section{Complementary moments}

The quantities $v_{k;m} = \int_{\Ifo{0}{1}}(1-x)^m\mu_k(\dx)$ satisfy,
according to \cite[\S7, Eq.\@ (7) \& (8)]{burnolirwin}, the following
recurrences. For $k\geq1$:
\begin{equation}\label{eq:recurrv} 
(b^{m+1} - b +1) v_{k;m}=\sum_{j=1}^{m}\binom{m}{j}\gamma_j'v_{k;m-j} 
+ \sum_{j=0}^{m}\binom{m}{j}(d')^j v_{k-1;m-j} \,,
\end{equation}
with $d'=b-1-d$ and $\gamma_j' = \sum_{\substack{0\leq a<b\\a\neq d'}} a^j$. And for $k=0$:
\begin{equation}\label{eq:recurrv0}
(b^{m+1} - b +1)v_{0;m} = b^{m+1} + \sum_{j=1}^m \binom{m}{j}\gamma_j'v_{0;m-j}\,.
\end{equation}
Let $z_{k;m} = v_{k;m} - b/(m+1)$. In the calculations that follow, we will
use $\gamma_0' = b-1$ several times. First consider $k\geq1$:
\begin{equation} 
  \begin{split}
    b^{m+1} \Bigl(z_{k;m} + \frac{b}{m+1}\Bigr)=
    &\sum_{j=0}^{m}\binom{m}{j}\gamma_j'\Bigl(z_{k;m-j}+\frac{b}{m+1-j}\Bigr)
    \\
    &+ \sum_{j=0}^{m}\binom{m}{j}(d')^j
    \Bigl(z_{k-1;m-j}+\frac{b}{m+1-j}\Bigr).
  \end{split}
\end{equation}
And thanks to Eq.\@ \eqref{eq:sommebinom} we obtain:
\begin{equation}\label{eq:recurzmk}
(b^{m+1} - b +1) z_{k;m}=\sum_{j=1}^{m}\binom{m}{j}\gamma_j'z_{k;m-j}
+ \sum_{j=0}^{m}\binom{m}{j}(d')^j z_{k-1;m-j}\,.
\end{equation}
The sums over $j$ stop at $m-1$ since $z_{k;0} = 0$ for
all $k$.

For $k=0$, the calculations are slightly different, but again use
Eqs.\@ \eqref{eq:foo} and \eqref{eq:sommebinom}:
\begin{align}
b^{m+1} \Bigl(z_{0;m} + \frac{b}{m+1}\Bigr)
= b^{m+1} &+ \sum_{j=0}^m \binom{m}{j}\gamma_j'\Bigl(z_{0;m-j}+\frac {b}{m+1-j}\Bigr)\\
\notag
= b^{m+1} &+ \sum_{j=0}^m \binom{m}{j}\gamma_j'z_{0;m-j}
\\&+\frac b{m+1}\Bigl(b^{m+1}-(d'+1)^{m+1}+(d')^{m+1}\Bigr),
\\\notag
(b^{m+1} - b +1)z_{0;m} = b^{m+1} &+ \sum_{j=1}^m \binom{m}{j}\gamma_j'z_{0;m-j}
\\&-\frac b{m+1}\Bigl((d'+1)^{m+1}-(d')^{m+1}\Bigr).
\end{align}

Let $c=b^{-1}$. Recall that $d'=b-1-d$, so using Bernoulli numbers,
\begin{equation}
c^{j+1}\gamma_j' = \sum_{p=0}^j \binom{j}{p}B_p\frac{c^p}{j+1-p} - c(1-(1+d)c)^j\,,
\end{equation}
is also a polynomial in $c$:
\begin{equation}
c^{j+1}\gamma_j' =
\begin{cases}
1 - c & (j=0),
\\
\frac12 - \frac32 c + (d+1) c^2 & (j=1),
\\
\frac1{j+1} - \frac32 c + jc^2(\frac1{12}+1+d) + \dots &(j\geq2).
\end{cases}
\end{equation}
By the previous calculations
\begin{align}\notag
(1 - c^m + c^{m+1})z_{0;m} = 1 
&+ \sum_{j=1}^m \binom{m}{j}c^{m-j}c^{j+1}\gamma_j'z_{0;m-j} 
\\
&- \frac{(1-dc)^{m+1}-(1-(d+1)c)^{m+1}}{(m+1)c}
\\\notag
= \hphantom{1}
&\frac{(m+1)c - (1-dc)^{m+1} + (1-(d+1)c)^{m+1}}{(m+1)c} \\ 
&+ \sum_{j=1}^m \binom{m}{j}c^{m-j}c^{j+1}\gamma_j' z_{0;m-j}\,.
\end{align}
As in \cite{burnollargeb}, it follows that the $z_{0;m}$ are
rational functions of $c$, which have their poles outside the open disk
$D(0,\rho)$, $\rho^2(\rho+1) = 1$, $\rho\approx\np{0.755}$. Furthermore,
$z_{0;0} = 0$ and
\begin{equation}
z_{0;1} = (d+\frac12)\frac{c}{1 - c + c^2}\;.
\end{equation}
In the recurrence relation, since $z_{0;m} = 0$, the summation actually
involves $1\leq j\leq m-1$. Therefore, by recurrence, $z_{0;m}$ does
not have a constant term in its power expansion as a function of $c$. Let us write
$z_{0;m} = \alpha_{0;m} c + y _m c^2 + \dots$. In the summation over $j$, the
terms in $j\leq m-2$ are $O(c^3)$, so it suffices to look only at the
contribution of the term with $j=m-1$, and it contributes only to $c^2$:
\begin{align}\notag
  &(1 - c^m + c^{m+1})\bigl(\alpha_{0;m} c + \beta_{0;m} c^2 + \dots\bigr)
\\\notag
={}
&
  \frac{(m+1)c - (1-dc)^{m+1} + (1-(d+1)c)^{m+1}}{(m+1)c}
\\  &+mc\Bigl(\frac1{m-1+1} + O(c)\Bigr)z_{0;1}
  + O(c^3)
\\
={}
  &\frac{(2d+1)m}2 c - \frac{(3d^2+3d+1)m(m-1)}{6} c^2 
+(d+\frac12)c^2 + O(c^3).
\end{align}
This gives for $m\geq2$:
\begin{align}
  \alpha_{0;m} &= \frac{(2d+1)m}2\\
  \beta_{0;m} &= d+\frac12 - \frac{(3d^2+3d+1)m(m-1)}{6}\;.
\end{align}
It turns out that these formulas are also valid for $m=1$.
Therefore, for $k=0$ and $m\geq1$:
\begin{equation}
v_{0;m} = \frac{b}{m+1} + \frac{(2d+1)m}2 b^{-1} -
\Bigl(\frac{(3d^2+3d+1)m(m-1)}{6}-(d+\frac12)\Bigr) b^{-2} + O_m(b^{-3}).
\end{equation}

We will now obtain a similar result for the
$z_{k;m}=v_{k;m}-b/(m+1)$, $k\geq1$. First,
using the notation from the previous section, for all $k$,
$u_{k;1}+v_{k:1} = b$, therefore $z_{k;1} = w_{k;1}$ is given by the explicit formula:
\begin{equation}
z_{k;1} = (d+\frac12)\frac{c^{2k+1}}{(1-c+c^2)^{k+1}}
= (d+\frac12)c^{2k+1} + (d+\frac12)(k+1) c^{2k+2} + O_k(c^{2k+3}).
\end{equation}
We will show by induction on $k$ and $m$ that more generally, we have the
following analytical expansions in the neighborhood of $c=0$:
\begin{equation}
 \forall k\geq0, m\geq1\qquad z_{k;m} = \alpha_{k,m} c^{2k+1} + \beta_{k,m} c^{2k+2} + O_{k,m}(c^{2k+3}).
\end{equation}
for coefficients that we will determine. We could use Proposition
\ref{prop:develu}
for this, but we will prove it directly.
We note that in Proposition \ref{prop:develu}, for $m\geq2$ the
multiplicity of $c$ in $w_{k;m}$ was $2k+2$, whereas here we will see that
the order of $c$ in $z_{k;m}$ is $2k+1$. We have already shown
that $z_{k;m}$ vanishes at least at order $1$ at $c=0$ for $k=0$ and all
the $m$. And by the explicit formula for $z_{k;1}$ we know that the latter
vanishes at order $2k+1$ at $c=0$.

Let us assume true for $k'<k$ and all $m'$, and for $k'=k$ and $m'<m$, that
$z_{k',m'}$ vanishes at least to the order $2k'+1$ at the origin. According to
Eq.\@ \eqref{eq:recurzmk}:
\begin{equation}
\begin{split}
(1 - c^m + c^{m+1}) z_{k;m}=
  &\sum_{j=1}^{m-1}\binom{m}{j}c^{m-j}c^{j+1}\gamma_j'z_{k;m-j}
\\
&+
  \sum_{j=0}^{m-1}\binom{m}{j}c^{m-1-j}c^{2}z_{k-1;m-j}(1 - (d+1)c)^j\,.
\end{split}
\end{equation}
Recall that $d'= b- 1 -d$, hence $cd' = 1 - (d+1)c$. The second sum is by
recurrence $O(c^{2+2k-1})$. The first sum is, by recurrence, $c\times
O(c^{2k+1})$. We have therefore established at this stage that $z_{k;m}$
vanishes at $c=0$ at least to the order $2k+1$. Let us now look more precisely
for $\alpha_{k;m}$ and $\beta_{k;m}$. The factor $(1-c^m + c^{m+1})$ does not
contribute since $m\geq2$.
The first sum does not contribute to the term of order $c^{2k+1}$. For the second
sum, only its last term with $j=m-1$ contributes, and we obtain:
\begin{equation}
\alpha_{k;m} = m \alpha_{k-1;1} = m (d + \frac12).
\end{equation}
For the next term of order $c^{2k+2}$, we repeat the reasoning more precisely,
keeping the term with $j=m-1$ from the first sum and the terms in $j=m-1$ and
$j=m-2$ from the second sum. And we get:
\begin{align}\notag
  \alpha_{k;m} + \beta_{k;m} c ={}&
    m c \frac1m \alpha_{k;1} + m (\alpha_{k-1;1} + \beta_{k-1;1} c)(1 -(m-1)(d+1)c)
\\
  & + \binom{m}2 c \alpha_{k-1;2} + O(c^2),
\\\notag
  \beta_{k;m} = {}&\alpha_{k;1} + m \beta_{k-1;1} - m(m-1)(d+1)\alpha_{k-1;1} + \binom{m}2 \alpha_{k-1;2}
\\
    = {}& (d+\frac12)\Bigl(1+mk - m(m-1)d\Bigr).
\end{align}
These formulas for $\alpha_{k;m}$ and $\beta_{k;m}$ were established for
$m\geq 2$ (and $k\geq 1$). We can verify that they also work for $m =
1$. Let us combine our results into a proposition, expressed again using the
$v_{k;m}$:
\begin{proposition}\label{prop:develv}
For $k=0$ and $m\geq1$:
\begin{equation}
v_{0;m} = \frac{b}{m+1} + \frac{(d+\frac12)m}{b} -
\frac{(d^2+d+\frac13)\frac{m(m-1)}2-d-\frac12}{b^{2}} + O_m(b^{-3}).
\end{equation}
For $k\geq1$ and $m\geq1$:
\begin{equation}
v_{k;m} = \frac{b}{m+1} +\frac{(d+\frac12) m}{b^{2k+1}} + 
\frac{(d+\frac12)\Bigl(1+mk - m(m-1)d\Bigr)}{b^{2k+2}} + O_{k,m}(b^{-2k-3}).
\end{equation}
For $m=1$ we have the explicit formula $v_{k;1} = \frac{b}{2} +
(d+\frac12)\frac{c^{2k+1}}{(1-c+c^2)^{k+1}}$ valid for all $k\geq0$.
\end{proposition}

\section{Digamma}

In this section, we extend our previous work \cite{burnoldigamma} from 
$k=0$ to $k>0$. Let
$\psi(x)=\frac{\mathrm{d}}{\dx}\log\Gamma(x)$ be the digamma function. It
satisfies the functional equation \cite[1.7.1 (8)]{erdelyiI}:
\begin{equation}\label{eq:psi}
\frac1x = \psi(x+1) - \psi(x),
\end{equation}
therefore
\begin{align}\notag
I(b,d,k) = \int_{\Ifo{b^{-1}}{1}}\frac{\mu_k(\dx)}{x} = 
&\underbrace{\int_{\Ifo{b^{-1}}1} \psi(x+1)\mu_k(\dx)}_{I_k} 
\\ 
&-
\underbrace{\int_{\Ifo{b^{-1}}1} \psi(x)\mu_k(\dx)}_{J_k}\,.
\end{align}
We assume $k\geq1$ and examine first the quantity $J_k$. By the
integration lemma Eq.\@ \eqref{eq:lemint1}:
\begin{align} 
J_k &= \int_{\Ifo{0}{1}} \Un_{\Ifo{b^{-1}}{1}}(x)\psi(x)\mu_k(\dx)\\
&=\int_{\Ifo01}\frac1b \sum_{\substack{0<a<b\\ a\neq d}} \psi(\frac{a+x}b)\mu_k(\dx) 
+ \delta_{d>0}\int_{\Ifo01}\frac1b \psi(\frac{d+x}b)\mu_{k-1}(\dx).
\end{align}
There is no contribution of $a=0$ because of the presence of the function
$\Un_{\Ifo{b^{-1}}{1}}()$ which is zero at $\frac{a+x}b$, $0\leq x < 1$ if
$a=0$ and equals $1$ otherwise. We have the relation
(see \cite[1.7.1.\@ (12)]{erdelyiI} and \cite[Eq.\@ (9)]{burnoldigamma}):
\begin{equation}
\frac1b\sum_{1\leq a \leq b} \psi(\frac{a+x}b) = \psi(x+1) - \log(b),
\end{equation}
and thus
\begin{equation}
\frac1b\sum_{\substack{0<a<b\\a\neq d}} \psi(\frac{a+x}b)
= \psi(x+1) - \log(b) - \frac1b\psi(\frac{b+x}b) - \delta_{d>0}\frac1b\psi(\frac{d+x}b)\;,
\end{equation}
therefore, given $\mu_k(\Ifo{0}{1}) = b$:
\begin{align}\notag
J_k = \int_{\Ifo{0}{1}} \psi(x+1)\mu_k(\dx) -b \log(b) 
& - \frac1b\int_{\Ifo{0}{1}}\psi(\frac{b+x}b)\mu_k(\dx)
\\ \notag
& - \delta_{d>0} \frac1b\int_{\Ifo{0}{1}}\psi(\frac{d+x}b)\mu_k(\dx)
\\\label{eq:J} 
& + \delta_{d>0} \frac1b\int_{\Ifo{0}{1}}\psi(\frac{d+x}b)\mu_{k-1}(\dx).
\end{align}
Moreover
\begin{equation} 
I_k 
= \int_{\Ifo{0}{1}} \psi(x+1)\mu_k(\dx)-\int_{\Ifo{0}{b^{-1}}} \psi(x+1)\mu_k(\dx),
\end{equation}
and, distinguishing the cases $d=0$ and $d>0$:
\begin{equation} 
\int_{\Ifo{0}{b^{-1}}} \psi(x+1)\mu_k(\dx)=\int_{\Ifo{0}{1}}\frac1b\psi(\frac{x}{b}+1)
\begin{cases} 
\mu_k(\dx)& (d> 0),
\\ 
\mu_{k-1}(\dx)& (d=0).
\end{cases}
\end{equation}
So we get
\begin{equation} 
\label{eq:I} 
I_k = \int_{\Ifo{0}{1}} \psi(x+1)\mu_k(\dx)-\frac1b\int_{\Ifo{0}{1}} \psi(\frac {b+x}b)
\begin{cases} 
\mu_k(\dx)& (d> 0),
\\ 
\mu_{k-1}(\dx)& (d=0).
\end{cases}
\end{equation}
By combining Eqs.\@ \eqref{eq:I} and \eqref{eq:J} it comes:
\begin{align}
\notag
I_k - J_k = b \log(b) 
& +\delta_{d=0} \frac1b\int_{\Ifo{0}{1}}\psi(\frac{b+x}b)\mu_k(\dx)
\\\notag 
& -\delta_{d=0} \frac1b\int_{\Ifo{0}{1}}\psi(\frac{b+x}b)\mu_{k-1}(\dx)
\\\notag 
& +\delta_{d>0} \frac1b\int_{\Ifo{0}{1}}\psi(\frac{d+x}b)\mu_k(\dx)
\\\label{eq:preIk} 
& -\delta_{d>0} \frac1b\int_{\Ifo{0}{1}}\psi(\frac{d+x}b)\mu_{k-1}(\dx).
\end{align}
\begin{proposition}\label{prop:Ibdk_digamma}
Let $d_1 = d$ if $d>0$ and $d_1 = b$ if $d=0$. Then (keeping in mind that
the $\mu_k$ depend on $b$ and $d$), for $k\geq1$:
\begin{equation}\label{eq:Ik}
I(b,d,k) = b \log(b) + \frac1b\int_{\Ifo{0}{1}}\psi(\frac{d_1+x}b)(\mu_k-\mu_{k-1})(\dx),
\end{equation}
and for $k=0$:
\begin{equation}\label{eq:I0}
I(b,d,0) = b\log(b) - \psi(1) + \frac1b\int_{\Ifo{0}{1}}\psi(\frac{d_1+x}b)\mu_0(\dx).
\end{equation}
\end{proposition}
\begin{proof}
The case $k\geq1$ is Eq.\@ \eqref{eq:preIk} and the case $k=0$ is \cite[Thm. 1]{burnoldigamma}.
\end{proof}
In this paper, we are primarily interested in the asymptotic expansion for
$b\to\infty$, with $d$ fixed. Therefore, the formulas above are significantly
more advantageous in the case $d=0$ because the argument of the digamma
function $\psi(t)$ is found in the neighborhood of $t=1$, whereas for a fixed
$d>0$, it is found in the neighborhood of the pole at $t=0$. We therefore
reformulate the proposition for $d>0$ in a directly equivalent but potentially
more useful form:
\begin{proposition}\label{prop:digammad+}
Let $d>0$. For $k\geq1$: 
\begin{align}\notag
I(b,d,k) = b \log(b) 
&+ \int_{\Ifo{0}{1}}\frac{(\mu_{k-1}-\mu_{k})(\dx)}{d + x} 
\\\label{eq:Ikd+} 
 &+ \frac1b\int_{\Ifo{0}{1}}\Bigl(\psi(1+\frac{d+x}b)-\psi(1+\frac db)\Bigr)
                          (\mu_{k}-\mu_{k-1})(\dx).
\end{align}
For $k=0$ (and with $\mu_\infty(\dx)=b\dx$):
\begin{align}\notag
I(b,d,0) = b\log(b) 
&- b \log(1+\frac1d) + \psi(1+\frac db)-\psi(1)
\\\notag
&+ \int_{\Ifo{0}{1}}\frac{(\mu_\infty- \mu_0)(\dx)} {d+x}
\\\label{eq:I0d+} 
&+\frac1b\int_{\Ifo{0}{1}}\Bigl(\psi(1+\frac{d+x}b)-\psi(1+\frac db)\Bigr)\mu_0(\dx).
\end{align}
Recall here that the
measures $\mu_k$, $k\geq0$, depend on $b$ and $d$.
\end{proposition}
\begin{proof}
  Direct consequence of the functional equation $\psi(x) = -x^{-1}+\psi(x+1)$
  \cite[1.7.1 (8)]{erdelyiI} and the fact that $\mu_k$ has a total mass equal
  to $b$ independent of $k$. We can also use only $\psi(1)$, but the form
  chosen with $\psi(1+\frac db)$ is more convenient for the following
  sections; this is also the reason for the appearance of $\mu_\infty$.
\end{proof}

\section{\texorpdfstring{$d=0$ and $k=0$}{d=0 and k=0}}

In this section we assume $d=0$ and $k=0$. The formula Eq.\@ \eqref{eq:I0} is
therefore with $d_1 = b$ and gives:
\begin{equation}
I(b,0,0) = b\log(b) + \frac1b\int_{\Ifo{0}{1}}\Bigl(\psi(1+\frac{x}b)-\psi(1)\Bigr)\mu_0(\dx).
\end{equation}
The following Taylor series expansion (\cite[1.17 (5)]{erdelyiI}) converges normally for
$|u-1|\leq1-\eta$, $\eta>0$:
\begin{equation}\label{eq:psiseries}
\psi(u) - \psi(1) = \sum_{m=1}^\infty (-1)^{m-1}\zeta(m+1)(u-1)^{m}\,.
\end{equation}
Therefore with $u=1+\frac xb$, $0\leq x <1$, and recalling
$u_{0;m}=\int_{\Ifo{0}{1}} x^m\mu_0(\dx)$:
\begin{equation}
K(b,0) = b\log(b) +\sum_{m=1}^\infty (-1)^{m-1}\frac{\zeta(m+1)u_{0;m}}{b^{m+1}}\;.
\end{equation}
Using the trivial upper bounds $0\leq u_{0;m}\leq b/(m+1)$ and $1\leq
\zeta(m+1)\leq \zeta(2)$ and the fact that $u_{0;m}$ is a rational
function of $c=b^{-1}$ with a simple pole at the origin, we see that, up
to $O(b^{-M})$, for any $M$, the expression is given by a partial sum
which is a rational fraction in $b$, regular (and even zero) at the
origin as a function of $c=b^{-1}$. Hence the assertion:
\emph{$I(b,0,0)-b\log(b)$ admits an asymptotic expansion $a_1/b +
  a_2/b^2 + \dots$ to all orders in inverse powers of $b$ as
  $b\to\infty$}. To obtain some terms of this expansion, it is
convenient to write $u_{0;m} = \frac{b}{m+1}-w_{0;m}$:
\begin{equation}
\sum_{m=1}^\infty (-1)^{m-1}\frac{\zeta(m+1)u_{0;m}}{b^{m+1}} =
\sum_{m=1}^\infty (-1)^{m-1}\frac{\zeta(m+1)}{(m+1)b^m}
- \sum_{m=1}^\infty (-1)^{m-1}\zeta(m+1)\frac{w_{0;m}}{b^{m+1}}\;.
\end{equation}
The first sum is $-\int_0^1(\psi(1)-\psi(1+\frac xb))\dx =
b\log\Gamma(1+\frac 1b)-\psi(1)$. In the second sum, we know exactly the
first term $w_{0;1} = \frac c2(1 - c+ c^2)^{-1} = \frac{b}{2(b^2-b+1)}$.
We can also give an exact formula for $w_{0;2}$, but it becomes
complicated. In any case, we know from Proposition \ref{prop:develu}
that it equals $\frac56 b^{-2}+O(b^{-4})$. All subsequent $w_{0;m}$ are
$O(b^{-2})$, so the term in $m=2$ contributes a $\frac56 \zeta(3)
b^{-5}$ which cannot be altered subsequently. This gives us the
following Proposition:
\begin{proposition}\label{prop:develKb0}
The deviation from $b\log(b)$ of the Kempner harmonic sum
$\sum'\frac1m=K(b,0)$, where only those integers not having the digit
 $0$ in their base-$b$ representation contribute,
 admits an asymptotic expansion in inverse powers of $b$ to all
orders. There holds:
\begin{equation*} 
K(b,0) 
\begin{aligned}[t] 
&= b\log(b) + b\log\Gamma(1+ \frac 1b) - \psi(1) 
- \frac{\zeta(2)}{2b(b^2-b+1)} + \frac{5\zeta(3)}{6b^5} + O(b^{-6})\\ 
&=b\log(b)+\frac{\zeta(2)}{2b}-\frac{\zeta(3)}{3b^2} 
\begin{aligned}[t]
  &- \frac{2\zeta(2)-\zeta(4)}{4b^3}
  \\
  &-\frac{5\zeta(2)/2+\zeta(5)}{5b^4}+\frac{5\zeta(3)+\zeta(6)}{6b^5} +
  O(b^{-6}).
\end{aligned}
\end{aligned}
\end{equation*}
\end{proposition}
\begin{proof}
The first line was justified previously, and the second follows from it
via Eq.\@ \eqref{eq:psiseries} (which can be integrated to obtain the
expansion of $\log\Gamma(1+x)$).
\end{proof}
Here is some numerical data. If we evaluate the second line at $b=10$,
we obtain $\approx\np{23.103447618168193}$ and (using \cite{baillie2008}
if we have access to \textsf{Mathematica\texttrademark} or the
\textsf{SageMath} code of the author \cite{burnolirwin}) we have
$K(10,0)\approx\np{23.103447909420542}$. For $b=1000$, the second line
gives the value $\approx \np{6907.7561010479319268743516533}$, while
using the author's code, we find
\begin{equation*}
  K(1000,0)\approx
  \np{6907.7561010479319268744907724}.
\end{equation*}
In this case, the formula in the
first line actually gives an approximation roughly three times further
from $K(1000,0)$ and slightly larger instead of smaller.

\section{\texorpdfstring{$d=0$ and $k\geq1$}{d=0 and k>=1}}

Still $d=0$ but now with $k\geq1$. This time our starting point
Eq.\@ \eqref{eq:Ik} is rewritten as:
\begin{align} 
I(b,0,k) &= b \log(b) - \frac1b\int_{\Ifo{0}{1}} 
\Bigl(\psi(1)-\psi(1 + \frac{x}b)\Bigr)(\mu_{k}-\mu_{k-1})(\dx)
\\
&=b \log(b) + \sum_{m=1}^\infty (-1)^{m-1}\frac{\zeta(m+1)(u_{k;m}-u_{k-1;m})}{b^{m+1}}
\\\label{eq:Ib0k}
&=b \log(b) + \sum_{m=1}^\infty (-1)^{m-1}\frac{\zeta(m+1)(w_{k-1;m}-w_{k;m})}{b^{m+1}}\;.
\end{align}
Each term in absolute value is bounded above by $\zeta(2)(m+1)^{-1}b^{-m}$ as
$0\leq u_{k-1;m}<u_{k;m} < \frac{b}{m+1}$. Therefore, up to $O(b^{-M})$ for
any $M$, the difference $I(b,d,k)-b\log(b)$ is a rational fraction in $b$, and
considering Proposition \ref{prop:develu} and
$u_{k;m}-u_{k-1;m}=w_{k-1;m}-w_{k;m}$, this rational fraction is regular in
$c=b^{-1}=0$ and even vanishes to the order $(2(k-1)+1)+2=2k+1$, this order of
vanishing being determined by the first term $m=1$, and in this first term by the
contribution of $w_{k-1;m}$, all other contributions vanishing to higher order
in the variable $c$.

For the term with $m=2$, we know that $w_{k-1;2}$ is $O(b^{-2k})$, therefore,
in the final term, a contribution $O(b^{-2k-3})$, and we know the first two
terms. We could, of course, calculate more, but we will be satisfied with
these, and this means that we can determine the first four terms of the
expansion of $I(b,0,k)-b\log(b)$, which will be of the form $a_1 b^{-2k-1} +
a_2 b^{-2k-2} + a_3 b^{-2k-3} + a_4 b^{-2k-4}$. It is important to note that
the term with $m=3$ will contribute to $b^{-2k-4}$.

For $m=1$, we have, by Proposition \ref{prop:develu}:
\begin{equation}\label{eq:diffwk1}
w_{k-1;1}-w_{k;1} = \frac{c^{2k-1}}{2(1-c+c^2)^k}\left(1 - \frac{c^2}{1 - c+ c^2}\right)
= \frac{c^{2k-1}(1-c)}{2(1-c+c^2)^{k+1}}\;.
\end{equation}
For $m=2$, we must distinguish between the cases $k=1$, $k=2$, and $k\geq3$.
Since we only consider the first two terms of $w_{k-1;2}$, the contribution of
$w_{k;2}$ is negligible.
\begin{align*} 
w_{0;2}-w_{1;2} &= \frac56 c^2 + O(c^4)\\ 
w_{1;2}- w_{2;2}&= \frac12 c^4 + \frac43 c^5 + O(c^6)\\ 
(k\geq 3)\quad w_{k-1;2} - w_{k;2} &= \frac12 c^{2k} +\frac{k}2 c^{2k+1} + O(c^{2k+2}).
\end{align*}
We must multiply this by $-\zeta(3)c^3$. The term with $m=3$ contributes
$+\frac12\zeta(4)b^{-2k-4}$.
\begin{align}\notag 
    I(b,0,k) = b\log(b) 
    + \frac{\zeta(2)(1-c)c^{2k+1}}{2(1-c+c^2)^{k+1}}
      &-
      \zeta(3)\begin{pmatrix}
        \frac56 c^5 & \textit{if } k=1
        \\
        \frac12 c^7 + \frac43 c^8 & \textit{if } k=2
        \\
        \frac12  c^{2k+3} + \frac{1}2 k c^{2k+4} & \textit{if } k\geq3
      \end{pmatrix}
      \\
      & + \frac12\zeta(4) c^{2k+4}+O(c^{2k+5}).
\end{align}
We thus obtain:
\begin{proposition}\label{prop:develIb0k}
  Let $b>1$ and $k\geq1$.  The deviation from $b\log(b)$ of the Irwin harmonic
  sum $I(b,0,k)=\sum^{(k)} \frac1m$, where the integers have exactly $k$ digits
  equal to $0$ in their base $b$ representation, admits an asymptotic expansion
  in inverse powers of $b$ to all orders. It depends on whether $k=1$,
  $k=2$, or $k\geq3$:
\begin{align}\notag
I(b,0,1) = b\log(b) + \frac{\zeta(2)}{2b^{3}} 
+ \frac{\zeta(2)}{2b^{4}}
&- \frac{3\zeta(2)+5\zeta(3)}{6b^{5}} 
\\
&- \frac{3\zeta(2)-\zeta(4)}{2b^{6}} +
  O(b^{-7}),
\\\notag
I(b,0,2)= b\log(b) + \frac{\zeta(2)}{2b^{5}} 
+ \frac{\zeta(2)}{b^{6}} 
&- \frac{\zeta(3)}{2b^{7}}
\\
&- \frac{15\zeta(2)+8\zeta(3)-3\zeta(4)}{6b^{8}} 
+ O(b^{-9}),
\end{align}
and, for $k\geq3$:
\begin{align}
\notag
I(b,0,k){}= b\log(b)
+ \frac{\zeta(2)}{2b^{2k+1}} 
+ \frac{k\zeta(2)}{2b^{2k+2}} 
&+ \frac{(k^2-k-2)\zeta(2)-2\zeta(3)}{4b^{2k+3}}
\\\notag
&
+\frac{(k^3-3k^2-10k-6)\zeta(2)-6k\zeta(3)+6\zeta(4)}{12b^{2k+4}}
\\
&
+ O(b^{-2k-5}).
\end{align}
\end{proposition}
Here are some numerical results. Let $\delta_k(b)$ be equal to $I(b,0,k)$
minus the approximation given in the Proposition. We obtain:\newline
\begin{equation*}\arraycolsep10pt
\begin{array}{c|c|c|c}
k&\delta_k(10)\cdot 10^{2k+4}&\delta_k(100)\cdot 100^{2k+4}&\delta_k(1000)\cdot 1000^{2k+4}
\\\hline
1&\np{-0.134}&\np{-0.019}&\np{-0.00195}
\\
2&\np{-0.497}&\np{-0.054}&\np{-0.00540}
\\
3&\np{-1.369}&\np{-0.137}&\np{-0.01365}
\\
4&\np{-2.446}&\np{-0.229}&\np{-0.02275}\\\hline
\end{array}
\end{equation*}

\section{\texorpdfstring{$d>0$ and $k=0$}{d>0 and k=0}}

From now on, $d>0$. Let us first assume $k=0$. The starting point is
Eq.\@ \eqref{eq:I0d+}, which we write:
\begin{equation}
K(b,d) = b\log(b) - b\log(1+\frac1d) + \psi(1+\frac db)-\psi(1) + A + B,
\end{equation}
with $A$ the term from the second line of Eq.\@ \eqref{eq:I0d+} and $B$ the term
from the last line. We first examine the quantity $B$:
\begin{equation}
B = \frac1b\int_{\Ifo{0}{1}}\Bigl(\psi(1+\frac{d+x}b)-\psi(1+\frac
db)\Bigr)\mu_0(\dx).
\end{equation}
The function $x\mapsto\psi(1+\frac{d+x}b)-\psi(1+\frac db)$ is analytic at
$x=0$ and the Taylor series there has a radius of convergence $d+b$ (since the
nearest singularity is at $x=-d-b$), so it converges normally on $\Ifo01$ and,
after recalling that $\int_{\Ifo01} x^m\mu_0(\dx) = u_{0;m}$, we obtain:
\begin{equation}
B = \sum_{m=1}^\infty \frac{1}{m!}\psi^{(m)}(1+\frac{d}{b})\frac{u_{0;m}}{b^{m+1}}\;.
\end{equation}
For all $m\geq1$ (\cite[1.16 (9)]{erdelyiI}):
\begin{equation} 
\frac{1}{m!}\psi^{(m)}(x) = (-1)^{m-1}\sum_{n=0}^\infty\frac{1}{(n+x)^{m+1}}\;,
\end{equation}
and therefore (with $c=b^{-1}$ in the second line):
\begin{align}\notag
B &= \sum_{m=1}^\infty (-1)^{m-1}\left(\sum_{n=1}^\infty \frac{1}{(bn +
    d)^{m+1}}\right)u_{0;m}
\\\label{eq:Bserie}
&= \sum_{m=1}^\infty (-1)^{m-1}\left(\sum_{n=1}^\infty \frac{1}{(n + dc)^{m+1}}\right)(cu_{0;m})c^m.
\end{align}
It is important to note, when considering this expression from the perspective
of analytic functions, that we know $c u_{0;m}<(m+1)^{-1}$, but only for
$c=b^{-1}$, $b\in\NN$, $b\geq2$. Therefore, initially, we require that $c$
among these specific ``inverse of integers'' values.

Except for $b=2$ and $d=1$, in which case the $u_{0;m}$, $m\geq1$, are all
zero, the sequence $(u_{0;m})$ is strictly positive and strictly decreasing
with a limit of zero. It follows that the series expressing $B$ is always a
special alternating series, and the difference between $B$ and a partial sum
is less than the absolute value of the first ignored term. We have the
immediate upper bound $\sum_{n=1}^\infty \frac{1}{(bn +
  d)^{m+1}}<b^{-m-1}\zeta(m+1)$, and we know that $u_{0;m}<\frac{b}{m+1}$, so
the $m$th term is bounded above by $b^{-m}\zeta(m+1)/(m+1)$. Thus, up to
$O(b^{-M-1})$, for any $M$, we can approximate $B$ (for our $b\in\NN$,
$b\geq2$) by the partial sum with $M$ terms.

We know that $u_{0;m}$ is a rational function in $c$ having a simple pole at
$c=0$ and no other pole in the open disk $D(0,\rho)$, $\rho^2(1+\rho)=1$,
$\rho\approx\np{0.755}$. The expression $\sum_{n=1}^\infty 1/(n+cd)^{m+1}$ is
an analytic function of $c$ in the complex plane excluding
$\{-d^{-1},-2d^{-1},-3d^{-1},\dots\}$, and $b^{-1}$ is indeed located in the
open disk $D(0,d^{-1})$. Like $b\geq2$, $b^{-1}$ is also located in the disk
$D(0,\rho)$ on which we know that $cu_{0;m}$ has no poles. We therefore have,
for each term indexed by $m\leq M$, at least a Taylor series expansion to
within $O(c^{M+1})$ (or more), and therefore also for the partial sum
including the first $M$ terms of Eq.\@ \eqref{eq:Bserie}. Thus, again for $b\in\NN$,
$b\geq2$:
\begin{align}\notag
B
&= \sum_{m=1}^M (-1)^{m-1}c^m(c u_{0;m})\sum_{n=1}^\infty \frac{1}{(n + cd)^{m+1}}
+ O(c^{M+1})
\\
&=\frac12\zeta(2) c + a_2 c^2 + \dots + a_{M} c^{M} + O(c^{M+1}).
\end{align}
Therefore, $B$ has, as a function of $b\in\NN$, $b\geq2$, an asymptotic
inverse power expansion to all orders with respect to $b$.

We can extract a part of $B$ by writing $u_{0;m}=\frac
b{m+1}-w_{0;m}$. Reversing the steps, this will correspond to the evaluation
of:
\begin{equation}
\int_0^1\Bigl(\psi(1+\frac{d+x}b)-\psi(1+\frac db)\Bigr)\dx
= b \log\frac{\Gamma(1+\frac{d+1}b)}{\Gamma(1+\frac db)} - \psi(1+\frac db).
\end{equation}
and therefore:
\begin{equation}
B = b \log\frac{\Gamma(1+\frac{d+1}b)}{\Gamma(1+\frac db)} - \psi(1+\frac db)
- \sum_{m=1}^\infty (-1)^{m-1}c^{m+1}w_{0;m}\sum_{n=1}^\infty \frac{1}{(n + cd)^{m+1}}\;.
\end{equation}
We know from Proposition \ref{prop:develu} that each $w_{0;m}$ for $m\geq2$ is
$O(c^2)$ and we have the exact formula $w_{0;1} =
(d+\frac12)c/(1-c+c^2)$. Therefore, among the terms with $m\geq2$, only the
one with $m=2$ will contribute to $c^5$. Thus, with regard to the term with
$m=1$, we will only need $\sum(n+cd)^{-2}$ to order $O(c^3)$ to determine $B$
to within $O(c^6)$. We obtain the following formula:
\begin{align}\notag
B = b \log\frac{\Gamma(1+\frac{d+1}b)}{\Gamma(1+\frac db)}
- \psi(1+\frac db)
&-
(d+\frac12)\frac{c^3}{1-c+c^2}\bigl(\zeta(2)-2d\zeta(3)c+3d^2\zeta(4)c^2\bigr)
\\
&+ (d^2+2d+\frac56)\zeta(3)c^5 + O(c^6).
\end{align}

Now let us consider the quantity $A =
\int_{\Ifo{0}{1}}\frac{(\mu_{\infty}-\mu_0)(\dx)}{d+x}$ extracted from
Eq.\@ \eqref{eq:I0d+} (where $\mu_{\infty}(\dx) = b\dx$). To analyze it
from the perspective of letting $b$ go to $\infty$, we will use
Eq.\@ \eqref{eq:U0rec} for $U_0(z)=\int_{\Ifo{0}{1}}(z+x)^{-1}\mu_0(\dx)$:
\begin{equation}
U_0(z) = \frac1z + \sum_{\substack{0\leq a <b\\a\neq d}}U_0(bz+a).
\end{equation}
Let us define
\begin{equation}
U_\infty(z) = \int_{\Ifo{0}{1}}\frac{b\dx}{z+x} = b\log(1+\frac1z).
\end{equation}
This function satisfies the following relation (which is the limit for
$k\to\infty$ of Eq.\@ \eqref{eq:Ukrec}):
\begin{equation} 
U_\infty(z) = \sum_{0\leq a <b} U_\infty(bz+a).
\end{equation}
So
\begin{align} 
A &= -\frac1d + U_\infty(bd+d)
+ \sum_{\substack{0\leq a <b\\a\neq d}}(U_{\infty}(bd+a)-U_0(bd+a))
\\
&=-\frac1d + b \log(1 + \frac1{bd+d})
+ \sum_{\substack{0\leq a <b\\a\neq d}}\int_{\Ifo{0}{1}}\frac{(\mu_{\infty}-\mu_0)(\dx)}{bd+a+x}
\\
&=-\frac1d + b \log(1 + \frac1{bd+d})
+ \sum_{\substack{0\leq a <b\\a\neq d}} 
\sum_{m=0}^\infty (-1)^{m}\frac{\mu_{\infty}(x^m )-\mu_0(x^m)}{(bd+a)^{m+1}}
\\
&=-\frac1d + b \log(1 + \frac1{bd+d})
- \sum_{m=1}^\infty(-1)^{m-1} w_{0;m}\sum_{\substack{0\leq a <b\\a\neq d}}\frac1{(bd+a)^{m+1}}
\\
\notag
&=-\frac1d + c^{-1}\log(1+\frac{c}{d(1+c)})
\\
&\phantom{{}=}
- \sum_{m=1}^\infty(-1)^{m-1} c^{m}w_{0;m}
  \sum_{\substack{0\leq a <b\\a\neq d}}\frac c{(d+ca)^{m+1}}\;.
\end{align}
We have the rough upper bound $\sum_{\substack{0\leq a <b\\a\neq
    d}}\frac1{(bd+a)^{m+1}}\leq b(bd)^{-m-1} = d^{-m-1}b^{-m}$ and we know
$0\leq w_{0;m}\leq b/(m+1)$. Therefore, if we keep $M$ terms of the series,
the remainder is bounded in absolute value by $\sum_{m=M+1}^\infty
d^{-m-1}b^{1-m}=d^{-M-2}b^{-M}/(1-1/(bd))\leq 2d^{-M-2}b^{-M}$. Consequently,
we can, up to $O(b^{-M})$, for each given $M$, simply look at what happens
with the partial sum with $M$ terms. And so it suffices to examine each
expression:
\begin{equation}
s_m(b,d) =
\sum_{\substack{0\leq a <b\\a\neq d}}\frac c{(d+ca)^{m+1}} = s_m^*(b,d) - c(d+cd)^{-m-1}\,,
\end{equation}
since we already know that the $c^m w_{0;m}$ are analytic (and vanish at the
origin at order $2$ for $m=1$, at order $m+2$ for $m\geq2$). It is well known
that the Riemann sum
\begin{equation}
  s_m^*(b,d) =
\frac1b \sum_{0\leq a <b}\frac 1{(d+\frac ab)^{m+1}}\approx \int_{d}^{d+1}\frac{dx}{x^{m+1}}\;,
\end{equation}
admits, by the Euler-MacLaurin method  (\cite[\S3.6]{debruijn1981},
\cite[\S8.1]{olver1997}), an asymptotic
expansion to all orders in inverse powers of $b$: in general, for any function
$f$ of class $C^\infty$ on $\Iff01$, we have the asymptotic expansion for
$b\to\infty$:
\begin{equation} 
\frac1b\sum_{a=0}^{b-1}f(\frac ab) =_{\text{asympt.}} \int_0^1f(x)\dx 
- \frac{f(1)-f(0)}{2b} 
+\sum_{i\geq1}\frac{B_{2i}}{(2i)! b^{2i}}\bigl(f^{(2i-1)}(1)- f^{(2i-1)}(0)\bigr),
\end{equation}
which can be applied here to the functions $f(x)=(d+x)^{-m-1}$. Therefore,
there also exists a series expansion for $s_m(b,d)$, and then for each partial
sum of the series in $A$, and we conclude that $A$, as a function of
$b\in\NN$, $b\geq2$, admits an asymptotic expansion at all orders in inverse
powers of $b$.

Later, we will need the expansion of $s_1(b,d)$ up to term $c^3$, here it is:
\begin{align}\notag
  s_1(b,d)
&= \frac1{d(d+1)}
\begin{aligned}[t]
  &+\frac c2(d^{-2}-(d+1)^{-2}) 
\\&+ \frac{c^2}{6}(d^{-3}-(d+1)^{-3}) -
  \frac{c}{d^2(1+c)^2} + O(c^4),
\end{aligned}
\\\label{eq:s1}
&
\begin{aligned}[b]
  {}=\frac1{d(d+1)}
  &-\frac12\Bigl({d}^{-2}+(d+1)^{-2}\Bigr) c\\
  &+ \frac16\Bigl({d}^{-3}- (d+1)^{-3}+12{d}^{-2}\Bigr) c^{2}-3d^{-2}c^3+O(c^4).
\end{aligned}
\end{align}
We will also need two terms from $s_2(b,d)$:
\begin{align}\notag
    s_2(b,d)
    &= \frac12\Bigl(d^{-2}-(d+1)^{-2}\Bigr) +\frac c2(d^{-3}-(d+1)^{-3})
   - \frac{c}{d^3(1+c)^3} + O(c^2)
    \\\label{eq:s2}
    &=\frac12\Bigl(d^{-2}-(d+1)^{-2}\Bigr) -\frac c2(d^{-3}+(d+1)^{-3})+O(c^2).
\end{align}
and of
\begin{equation}\label{eq:s3}
s_3(b,d) = \frac13\Bigl(d^{-3} - (d+1)^{-3}\Bigr) + O(c)
\end{equation}

In each truncation of $A$ to a partial sum, each term with $m\geq2$
contributes a principal term of order $c^{m+2}$. To obtain a result within
$O(c^6)$, we must keep the terms with $m=1$, $m=2$, $m=3$. And we only need
the two principal terms of $w_{0;2}$ and the principal term of $w_{0;3}$,
which we have already tabulated in the Proposition \ref{prop:develu}.

At this stage, we have therefore established that $K(b,d)-b\log(b)+
b\log(1+\frac1d)$ admits, for a fixed $d$, an asymptotic expansion to all
orders in inverse powers of $b$ for $b\to\infty$. And by combining our
results, we obtain (with $c=b^{-1}$):
\begin{align}\notag
K(b,d)={}
&b\log(b) - b \log(1+ \frac1d)
\\\notag
&+b\log\frac{\Gamma(1+\frac{d+1}b)}{\Gamma(1+\frac db)} - \psi(1)
\\\notag
&-(d+\frac12)\frac{c^3}{1-c+c^2}\Bigl(\zeta(2) - 2d \zeta(3) c+3d^2\zeta(4)c^2\Bigr)
+(d^2+2d+\frac56)\zeta(3)c^5
\\\notag
&-\frac1d + c^{-1}\log(1+\frac{c}{d(1+c)})
-(d+\frac12)\frac{c^2}{1-c+c^2}\sum_{\substack{0\leq a<b\\a\neq d}}\frac{c}{(d+ac)^2}
\\\notag
&+(d^2+2d+\frac56)c^4\sum_{\substack{0\leq a <b\\a\neq d}}\frac c{(d+ca)^{3}} 
-(d+\frac12)c^5\sum_{\substack{0\leq a <b\\a\neq d}}\frac c{(d+ca)^{4}}
\\
&+O(c^6).
\end{align}
With the help of \textsf{Maple\texttrademark} applied to the above formula
where we replaced the finite sums of inverse powers $s_1(b,d)$, $s_2(b,d)$,
and $s_3(b,d)$ with their previously given Euler-Maclaurin expansions, we
explicitly obtained the first five coefficients:
\begin{proposition}\label{prop:develKbd}
  Let $d>0$ and $b>d$ be two integers. The Kempner harmonic sum $K(b,d)=\sum'
  m^{-1}$ where the denominators are the non-zero natural numbers not having
  the digit $d$ in base $b$ is such that $K(b,d)- (b\log(b)-b\log(1+d^{-1}))$
  has, for $b\to\infty$ at a fixed $d$, an asymptotic expansion to all orders
  in inverse powers of $b$.
This development starts with:
\begin{equation*}
I(b,d,0) = b \log(b) - b\log(1+\frac1d)
+ \frac{a_1(d)}{b}
+ \frac{a_2(d)}{b^2}
+ \frac{a_3(d)}{b^3}
+ \frac{a_4(d)}{b^4}
+ \frac{a_5(d)}{b^5} + O(b^{-6}),
\end{equation*}
with:
\begin{align*}
a_1(d) &= (d+\frac12) (\zeta(2)-d^{-2}),
\\
a_2(d) &=-(d^2+d+\frac13)\zeta(3)+\frac{9d^2+8d+2}{6d^3(d+1)}\;,
\\ 
a_3(d) &=(d+\frac12)(d^2+d+\frac12)\zeta(4) 
-(d+\frac12)\zeta(2) 
-(d+\frac12)\frac{2d^4+6d^3+6d^2+4d+1}{2d^4(d+1)^2}\;,
\\ 
a_4(d) &= 
\begin{aligned}[t] 
&-(d^4+2d^3+2d^2+d+\frac15)\zeta(5) 
+(2d+1)d\zeta(3)-(d+\frac12)\zeta(2) 
\\ 
&+\frac{60d^7+240d^6+500d^5+705d^4+627d^3+331d^2+96d+12}{60d^5(d+1)^3} \;,
\end{aligned}
\\ 
a_5(d) &= 
\begin{aligned}[t] 
&(d+\frac12)(d^2+d+1)(d^2+d+\frac13)\zeta(6) 
-3(d+\frac12)d^2\zeta(4) 
\\
&+(3d^2+3d+\frac56)\zeta(3) 
\\ 
&-(d+\frac12)\frac{9d^6+38d^5+56d^4+44d^3+22d^2+7d+1}{3d^6(d+1)^3}\;.
\end{aligned}
\end{align*}

\end{proposition}
\begin{remark}
  If we substitute $d=0$ into the factors multiplying the $\zeta(n)$, we
  recover exactly the proposition \ref{prop:develKb0} describing the behavior
  of $K(b,0)$. But here, in each coefficient, there is an additional rational
  fraction in $d$ that has a pole at $d=0$ (and moreover, the series describes
  the deviation of $K(b,d)$ from $b\log(b)-b\log(1+d^{-1})$, not from
  $b\log(b)$).
\end{remark}
Regarding the numerical verification, let $\delta_5(b,d)$ be the difference
between $K(b,d)$ and the value given in the proposition up to and including
the term in $b^{-5}$. We sought to verify that $\delta_5(b,d)\cdot
b^{5}$ behaved indeed as a multiple of $b^{-1}$. This appears to be valid. In
the data that follows, we see that even when $d$ is still moderately small,
the coefficients can be large; therefore, it is not their absolute magnitude
that interests us, but rather their behavior when $b$ is multiplied by $10$ or
by $2$.
\begin{equation*}\arraycolsep4.5pt
\begin{array}{c|cccccc} 
(b,d)&(10,1)&(100,1)&(1000,1)&(10,2)&(100,2)&(1000,2) 
\\ 
b^5\delta_5(b,d)&\np{-0.0095}&\np{0.00336}&\np{0.000375} 
&\np{-22.984}&\np{-2.8907}&\np{-0.2967} 
\\\hline 
(b,d)&(10.3)&(100.3)&(1000.3)&(10.9)&(100.9)&(1000.9) 
\\ 
b^5\delta_5(b,d)&\np{-138.36}&\np{-18.364}&\np{-1.8986}&\np{-37410.9}&\np{-6651.1}&\np{-721.35} 
\\\hline 
(b,d)&(100,20)&(1000,20)&(100,10)&(200,10)&(400,10)&(800,10) 
\\
b^5\delta_5(b,d)&\np{-618258.2}&\np{-72949.1}&\np{-12042.9}&\np{-6321.0}&\np{-3241.1}&\np{-1641.5}
\end{array}
\end{equation*}
To give an example with absolute numerical values, here is the table of $a_i(1)$:
\begin{equation*}\arraycolsep5pt
\begin{array}{cn{2}{33}} 
i&\multicolumn{1}{c}{a_i(1)}
\\\hline
1&0.967401100272339654708622749969038
\\
2&-1.221466107372386665932722376860050
\\
3&-1.971188973855571436523608887939658
\\
4&0.066067527317549658236125718531700
\\
5&2.963203104060488388745299065364318 
\end{array}
\end{equation*}
and that of successive approximations of $I(1000,1,0)$:
\begin{equation*}\arraycolsep10pt
\begin{array}{cl} 
b\log(b)-b\log(1+\frac1d) 
& \np{6802.3947633243107508264733832137798}
\\ 
+a_1/b & \np{6802.4102729139951916936692538739123}
\\ 
+a_2/b^2& \np{6802.4101643465467548823809917399976}
\\
+a_3/b^3& \np{6802.4101652613168283745373057181486}
\\
+a_4/b^4& \np{6802.4101652530142343068832406562677}
\\
+a_5/b^5& \np{6802.4101652530915088178110862985709}
\\
I(1000,1,0) &\np{6802.410165253090787463765128313543}
\end{array}
\end{equation*}

Similarly, here is the table of coefficients $a_i(9)$ (We rounded so that they
all have the same number of digits after the decimal point):
\begin{equation*}\arraycolsep5pt
\begin{array}{cn{6}{30}}
i&\multicolumn{1}{c}{a_i(9)}
\\\hline
1&15.509589684440867195870660132520
\\
2&-108.567448436811288262133914700522
\\
3&914.770073492156313978150997011578
\\
4&-8302.594067654065061880891954936008
\\
5&77274.510927845642303169206910055693 
\end{array}
\end{equation*}
and that of successive approximations of $I(1000,9,0)$
\begin{equation*}\arraycolsep10pt
\begin{array}{cl}
b\log(b)-b\log(1+\frac1d) 
& \np{6802.3947633243107508264733832137798}
\\ 
+a_1/b& \np{6802.4102729139951916936692538739123}
\\ 
+a_2/b^2& \np{6802.4101643465467548823809917399976}
\\ 
+a_3/b^3& \np{6802.4101652613168283745373057181486}
\\ 
+a_4/b^4& \np{6802.4101652530142343068832406562677}
\\ 
+a_5/b^5& \np{6802.4101652530915088178110862985709}
\\ 
I(1000,9,0)& \np{6802.410165253090787463765128313543}
\end{array}
\end{equation*}
Overall, these numerical results therefore seem relatively consistent.

\section{\texorpdfstring{$d>0$ and $k>0$}{d>0 and k>0}}

Still with $d>0$, we now consider $k\geq1$. Let's first show the existence of
the asymptotic expansion. The starting point is
Eq.\@ \eqref{eq:Ikd+}. We follow the procedure already shown for $k=0$.
\begin{align} 
B &= \frac1b\int_{\Ifo{0}{1}}\Bigl(\psi(1+\frac{d+x}b)-\psi(1+\frac db)\Bigr) 
(\mu_{k}-\mu_{k-1})(\dx)
\\ 
&= \sum_{m=1}^\infty (-1)^{m-1}\left(\sum_{n=1}^\infty \frac{1}{(bn + d)^{m+1}}\right) 
(u_{k;m}-u_{k-1;m})
\\ 
&= \sum_{m=1}^\infty (-1)^{m-1}\left(\sum_{n=1}^\infty \frac{1}{(n + dc)^{m+1}}\right)
(w_{k-1;m}-w_{k;m})c^{m+1}\,.
\end{align}
Using the bound $0\leq cw_{k-1;m} - cw_{k;m}\leq (m+1)^{-1}$ known for
$c=b^{-1}$, $b\in\NN$, $b\geq2$, we then obtain, using exactly the same
arguments as those used in the previous section, the proof of the existence of
an asymptotic expansion in powers of $c$, for $c$ of the given form and
tending to zero.

Let us now turn to the other, more delicate, part of
Eq.\@ \eqref{eq:Ikd+}, given the information usable about the moments. Again,
we follow the arguments presented for $d>0$ and $k=0$ which call upon the
property Eq.\@ \eqref{eq:Ukrec} of the functions $U_k(z)$ (and Eq.\@ \eqref{eq:U0rec} must
also be used if $k=1$ for $U_{k-1} = U_0$).
\begin{align} 
A &= \int_{\Ifo{0}{1}}\frac{(\mu_{k-1}-\mu_{k})(\dx)}{d + x} 
= U_{k-1}(d) - U_k(d)
\\ 
&= \delta_{k=1}\frac1d+\delta_{k>1}U_{k-2}(bd+d) 
- U_{k-1}(bd+d)+\sum_{\substack{0\leq a <b\\ a\neq d}} (U_{k-1}(bd+a)-U_k(bd+a))
\\ 
&=\sum_{\substack{0\leq a <b\\ a\neq d}} 
\sum_{m=1}^\infty(-1)^{m-1}\frac{u_{k;m}- u_{k-1;m}}{(bd+a)^{m+1}} 
\begin{aligned}[t] 
&+\delta_{k>1} 
\sum_{m=1}^\infty (-1)^{m-1}\frac{u_{k-1;m}- u_{k-2;m}}{(bd+d)^{m+1}}\\ 
&+\delta_{k=1} \Bigl(\frac1{bd+d} + \sum_{m=1}^\infty 
(-1)^{m-1}\frac{u_{0;m}}{(bd+d)^{m+1}}\Bigr).
\end{aligned}
\end{align}
We started the series at $m=1$ because $u_{0;m}=b$ for all $m$. Let us record
the expressions obtained for $A$ and $B$ in the following expression which can
be used for numerical calculation:
\begin{proposition}\label{prop:Ibdk}
Let $0<d<b$ and $k\geq1$. Then $I(b,d,k)$ is
\begin{align}\notag
b\log(b) &+ \sum_{m=1}^\infty (-1)^{m-1}(u_{k;m}- u_{k-1;m})
\Bigl(\sum_{n=1}^\infty \frac{1}{(bn + d)^{m+1}}
+\sum_{\substack{0\leq a <b \\ \smash[b]{a\neq d}}} \frac1{(bd+a)^{m+1}}\Bigr)
\\\notag 
&+\delta_{k>1} \sum_{m=1}^\infty (-1)^{m-1}\frac{u_{k-1;m}- u_{k-2;m}}{(bd+d)^{m+1}}
\\
&+\delta_{k=1} \Bigl(\frac{1}{bd+d}
+ \sum_{m=1}^\infty (-1)^{m-1}\frac{u_{0;m}}{(bd+d)^{m+1}}\Bigr).
\end{align}
\end{proposition}
In this proposition, $B$ corresponds to the part of the coefficients of the
first series that is given by a series in $n\geq1$. We have already
explained that this gives a total contribution admitting an
asymptotic expansion.

The most complex term in the preceding equation is that arising from the
finite sums of inverse powers over numbers of the form $bd+a$, $a\neq d$. But
we already encountered a similar series in the previous section with $w_{0;m}$
instead of $u_{k;m}-u_{k-1;m}$, and the explanations given work identically.

When $k>1$, the second line gives a contribution of the same type as $B$, but
simpler. We will simply point out that it is important that the denominator is
not $d^{m+1}$, for example, but rather $(bd+d)^{m+1}$; otherwise, we would
need more precise information about the moments, in a controlled manner for
$m\to\infty$.

When $k=1$, the series in the second row with coefficients $u_{0;m}$ is again
easily analyzed by invoking the upper bound by $b/(m+1)$: here again, for any
$M$, a finite number of terms suffices to understand the sum up to
$O(b^{-M})$. We are then reduced to a rational expression in $c=b^{-1}$,
regular at the origin.

We have thus established that $I(b,d,k)-b\log(b)$ admits, for $b\to\infty$, an
asymptotic expansion in inverse powers of $b$, and it remains for us to give
the first few terms. In view of Proposition \ref{prop:develIb0k}, we expect to
have to distinguish $k=1$, $k=2$, and $k\geq 3$.

Let us first consider $k=1$. We then have, with $c=b^{-1}$ and the notation
from the previous section:
\begin{align}
\notag
I(b,d,1) ={}&b\log(b)
\\
&\begin{aligned}[b]
&+\sum_{m=1}^\infty (-1)^{m-1}(w_{0;m}- w_{1;m})c^m
\Bigl(\sum_{n=1}^\infty \frac{c}{(n + cd)^{m+1}} +s_m(b,d)\Bigr)
\\
&+\frac{c}{d+dc} + \sum_{m=1}^\infty (-1)^{m-1}\frac{(cu_{0;m})c^{m}}{(1+c)^{m+1}d^{m+1}}\;,
\end{aligned}
\end{align}
\begin{align}
\notag
={}&b\log(b)
\\\label{eq:100}
&\begin{aligned}[b] 
&+\sum_{m=1}^\infty (-1)^{m-1}(w_{0;m}- w_{1;m})c^m
\Bigl(\sum_{n=1}^\infty \frac{c}{(n + cd)^{m+1}} +s_m(b,d)\Bigr)
\\ 
&+\frac{c}{d(1+c)} + c^{-1}\left(\frac{c}{d(1+c)}-\log(1+\frac{c}{d(1+c)})\right)
\\ 
&- \sum_{m=1}^\infty (-1)^{m-1}\frac{w_{0;m} c^{m+1}}{(1+c)^{m+1}d^{m+1}}\;.
\end{aligned}
\end{align}
We have already tabulated the start of the power expansions of $c$ of the
Riemann sums $s_1(b,d)$, $s_2(b,d)$, $s_3(b,d)$, which each have a finite limit
at $c=0$. Recall from Proposition \ref{prop:develu} that (note well that all $O()$
depend on $d$ and $m$):
\begin{align}
w_{0;1} - w_{1;1} &= (d+\frac12)\frac{c(1-c)}{(1-c+c^2)^2}\;,
\\
w_{0;2} - w_{1;2} &= (d^2+2d+\frac56)c^2 + O(c^4),
\\
w_{0;3} - w_{1;3} &= (d+\frac12) c^2 + (d+\frac12)d(d+1) c^3 + O(c^4),
\\
w_{0;m} - w_{1;m} &= (d+\frac12) c^2-(d+\frac12)(\frac{m}{2}-1)c^3 + O(c^4) \quad (m\geq4),
\\
w_{0;1} &= (d+\frac12)\frac{c}{1-c+c^2}\;,
\\
w_{0;2} &= (d^2+2d+\frac56)c^2 + O(c^4),
\\
w_{0;m} &= O(c^2)\quad (m\geq3),
\\
\sum_{n=1}^\infty \frac{c}{(n + cd)^{2}}&=\zeta(2)c - 2d\zeta(3)c^2+3d^2\zeta(4)c^3+O(c^4).
\end{align}
The first sum in Eq.\@ \eqref{eq:100} will give a contribution of order $c^2$ (more
precisely: equivalent to $(d+\frac12)(d(d+1))^{-1} c^2$). The principal term
contributed by the second sum is also given by $m=1$ and is
$\frac{c}{2}d^{-2}$. Given $c/(d+dc)$, we therefore have:
\begin{equation}
I(b,d,1) = b\log(b) + \frac{d+\frac12}{d^2} b^{-1} + O_d(b^{-2}).
\end{equation}
We had become accustomed to $\zeta(n)$ and here there are none of at this
order... Here is some numerical data obtained with the Python code by the
author \cite{burnolirwin} (its precision is sufficient here, but to study the
higher-order approximation, for $b=1000$, you must use the \textsf{SageMath}
implementation).
\begin{equation*}\arraycolsep5pt
\begin{array}{cn{1}{4}n{1}{4}n{1}{4}n{1}{4}n{1}{4}} 
&\multicolumn{1}{c}{1}&\multicolumn{1}{c}{2}&\multicolumn{1}{c}{3}
&\multicolumn{1}{c}{4}&\multicolumn{1}{c}{5}
\\
10&0.9211&0.9986&1.0687&1.1428&1.2217
\\
100&0.9897&0.9947&0.9972&0.9992&1.0012
\\
1000&0.9989&0.9994&0.9996&0.9997&0.9998
\end{array}
\end{equation*}
At the intersection of row $b$ and column $d$, we displayed
the rounded value of $(I(b,d,1)-b\log(b))bd^2/(d+\frac12)$.

We previously gave four terms of $s_1(b,d)$, two terms of $s_2(b,d)$, and the
principal term of $s_3(b,d)$ (Eqs.\ \eqref{eq:s1}, \eqref{eq:s2}, and
\eqref{eq:s3}). Combined with the series expansions given above, we can easily
obtain a result accurate to within $O(c^6)$. Here it is:
\begin{proposition}\label{prop:develIbd1}
  Let $d>0$ and $b>d$ be two integers. The Irwin harmonic sum
  $I(b,d,1)=\sum^{(1)} m^{-1}$ where the denominators are the natural numbers
  having exactly one digit $d$ in base $b$ is such that $I(b,d,1)- b\log(b)$
  has, for $b\to\infty$ at a fixed $d$, an asymptotic expansion to all orders
  in inverse powers of $b$. The start of this expansion is:
\begin{equation*}
I(b,d,1) = b \log(b)
+ \frac{a_1(d)}{b}
+ \frac{a_2(d)}{b^2}
+ \frac{a_3(d)}{b^3}
+ \frac{a_4(d)}{b^4}
+ \frac{a_5(d)}{b^5} + O(b^{-6}),
\end{equation*}
with:
\begin{align*}
a_1(d) &= (d+\frac12) d^{-2}\,,
\\
a_2(d) &=-\frac{9d^2+8d+2}{6d^3(d+1)};,
\\
a_3(d) &= (d+\frac12)\zeta(2)
+ (d+\frac12)\frac{2d^3+4d^2+4d+1}{2d^4(d+1)^2}\;,
\\ 
a_4(d) &= 
\begin{aligned}[t] 
&-(2d+1)d\zeta(3) +(d+\frac12)\zeta(2) 
\\ 
&-\frac{60 d^7+180 d^6+350 d^5+585 d^4+597 d^3+331 d^2+96d+12}{60d^5(d+1)^3}\;,
\end{aligned}
\\ 
a_5(d) &= 
\begin{aligned}[t] 
&+3(d+\frac12)d^2\zeta(4) 
-(3d^2+3d+\frac56)\zeta(3)
-(d+\frac{1}{2})\zeta(2)
\\
&-\frac{12d^8-24d^7-216d^6-387d^5-339d^4-186d^3-72d^2-18d-2}{12d^6(d+1)^3}\;.
\end{aligned}
\end{align*}
\end{proposition}
For the purpose of examining numerical validity, let $\delta_5(b,d)$ be the
difference between the value of $I(b,d,1)$ calculated using the SageMath code
by \cite{burnolirwin} and the approximation given by the asymptotic expansion
up to and including the order $O(b^{-5})$. The following table gives the
values of $b^6 \delta_5(b,d)$ for $d\in\{1,2,\dots,8\}$ and
$b\in\{125,250,500,1000\}$.
\begin{equation*}\arraycolsep3.2pt
\begin{array}{clllln{5}{2}n{5}{2}n{5}{2}n{6}{2}} 
&\multicolumn{1}{c}{1}&\multicolumn{1}{c}{2}&\multicolumn{1}{c}{3} 
&\multicolumn{1}{c}{4}&\multicolumn{1}{c}{5} 
&\multicolumn{1}{c}{6}&\multicolumn{1}{c}{7}&\multicolumn{1}{c}{8}
\\
125&-29.89&2.04&-125.04&-596.91&-1710.36&-3846.98&-7468.02&-13110.55
\\
250&-30.15&2.32&-125.75&-605.27&-174 4.25&-3943.24&-7691.90&-13566.66
\\
500&-30.28&2.47&-126.09&-609.52&-1761.67&-3993.12&-7808.76&-13806.41
\\
1000&-30.34&2.55&-126.27&-611.67&-1770.50&-4018.52&-7868.48&-13929.37
\end{array}
\end{equation*}
Here are the $b^5\delta_4(b,d)$ and the entries on the last line
are the values ofs $a_5(d)$:
\begin{equation*}\arraycolsep2.8pt
\begin{array}{cllllllln{4}{2}n{4}{2}n{4}{2}}
&\multicolumn{1}{c}{1}&\multicolumn{1}{c}{2}&\multicolumn{1}{c}{3}
&\multicolumn{1}{c}{4}&\multicolumn{1}{c}{5}
&\multicolumn{1}{c}{6}&\multicolumn{1}{c}{7}&\multicolumn{1}{c}{8} 
&\multicolumn{1}{c}{9}&\multicolumn{1}{c}{10}
\\
125&6.78&7.35&51.71&148.64&314.59&565.86&918.21&1386.80&1986.25&2730.59
\\
250&6.90&7.34&52.21&150.99&321.30&580.87&947.19&1437.42&2068.40&2856.64
\\
500&6.96&7.34&52.46&152.20&324.75&588.65&962.34&1464.08&2111.97&2923.93
\\
1000&6.99&7.33&52.58&152.80&326.50&592.62&970.09&1477.76&2134.41&2958.72
\\
\multicolumn{1}{c}{a_5(d)}&7.02&7.33&52.71&153.42&328.27&596.64&977.95&1491.69&2157.32&2994.32
\end{array}
\end{equation*}
It is interesting to note that the number $d=2$ seems to have a $a_6(d)$
significantly smaller than the others, and moreover, we see that
$b^5\delta_4(b,2)$ converges much faster to its limit $a_5(2)$ than what we
observe for $d\neq2$. All of this seems consistent since the lines above are
approximately $a_5(d)+a_6(d)/b$.

Let us move on to $k\geq2$. Expressed with $c=b^{-1}$ and $w_{k;m} =
\frac{b}{m+1}-u_{k;m}$, the formula from proposition \ref{prop:Ibdk} gives:
\begin{align}\notag
I(b,d,1) &= b\log(b)\\
\notag
&\phantom{{}=}+\sum_{m=1}^\infty (-1)^{m-1}(w_{k-1;m}- w_{k;m})c^m
\Bigl(\sum_{n=1}^\infty \frac{c}{(n + cd)^{m+1}} +s_m(b,d)\Bigr)
\\
&\phantom{{}=}+\sum_{m=1}^\infty (-1)^{m-1}(w_{k-2;m}- w_{k-1;m})\frac{c^{m+1}}{d^{m+1}(1+c)^{m+1}}\;.
\end{align}
The term with $m=1$ in the first sum will contribute, by Eq.\@ \eqref{eq:diffwk1}, a
$c^{2k-1+1}=c^{2k}$. While the term with $m=1$ in the second sum contributes,
for the same reason, a $c^{2k-3+2}=c^{2k-1}$. Therefore, the first term of the
expansion will be of order $c^{2k-1}$ (just as for $k=1$, it was or order $b^{-1}$). If we
want to give five coefficients, we must go up to $c^{2k+3}$ and therefore work
to within $O(c^{2k+4})$. In the first sum, the term indexed by $m\geq2$
contributes $c^{2k+m}$, so we keep $m=2$ and $m=3$. In the second sum, the
term indexed by $m\geq2$ contributes $c^{2k-1+m}$, so we must take into
account an additional $m=1$, $m=2,3,4$. We already know the first two terms of
$w_{k-2;2}- w_{k-1;2}$, but we need three, so we need the first three terms of
$w_{k-2;2}$.

Let us start with $k=2$, so we need $w_{0;2}$. The recurrence relation
Eq.\@ \eqref{eq:recurw} gives here
\begin{equation}
(1-c^2+c^3)w_{0;2} = 2(c^2\gamma_1)cw_{0:1}+((d+1)^3-d^3)\frac{c^2}{3}\;,
\end{equation}
with $c^2\gamma_1 = \frac12 - \frac12 c - dc^2$ and
$w_{0;1}=(d+\frac12)c/(1-c+c^2)$. We obtain:
\begin{equation}
w_{0;2} = (d^2+2d+\frac56)c^2 - (d^2-\frac13)c^4 + O(c^5).
\end{equation}
For $w_{1;2}$, the recurrence relation is Eq.\@ \eqref{eq:recurw12c}:
\begin{equation}
(1-c^2+c^3)w_{1;2} = 2c (c^2\gamma_1) w_{1;1} + c^3(w_{0;2} + 2 d w_{0;1}),
\end{equation}
with $w_{1;1} = (d+\frac12)c^3/(1-c+c^2)^2$, and we see that to obtain the
first three terms of $w_{1;2}$, we only need in this recurrence than the first
two (one of which is zero) of $w_{0;2}$. We obtain:
\begin{equation}
w_{1;2} = \frac12(2d+1)^2c^4+(3d^2+4d+\frac43)c^5+0\cdot c^6 + O(c^7),
\end{equation}
which is the formula of Proposition \ref{prop:develu} but we see that the
$O(c^6)$ was an $O(c^7)$.

The recurrence relation for $w_{2;2}$, which we already know is of order $c^{6}$, is:
\begin{equation}
(1-c^2+c^3)w_{2;2} = 2c (c^2\gamma_1) w_{2;1} + c^3(w_{1;2} + 2 d w_{1;1}).
\end{equation}
It must be evaluated to within $O(c^9)$, and therefore we need only the first
two terms of $w_{1;2}$, but three terms of $w_{1;1}$. As for
$w_{2;1}$, it is $(d+\frac12)c^5/(1-c+c^2)^3$, and we also need its first
three terms. To determine the third term of $w_{j;2}$, we will not need to
set-up an induction; it is determined by the first two, which we already know from
Proposition \ref{prop:develu}. For $w_{j;2}$ with $j=2$, we obtain:
\begin{equation}
w_{2;2} = \frac12(2d+1)^2 c^6+\frac32(2d+1)^2 c^7+(5d^2+6d+\frac{11}{6})c^8+O(c^9).
\end{equation}
For $j\geq3$, if we set
\begin{equation}
w_{j;2} = \frac12(2d+1)^2 c^{2j+2}+\frac{j+1}2(2d+1)^2c^{2j+3}+ \alpha_j c^{2j+4} + O(c^{2j+5}),
\end{equation}
then the relation determining $\alpha_j$ comes from
\begin{equation}
(1-c^2+c^3)w_{j;2} = 2c (c^2\gamma_1) w_{j;1} + c^3(w_{j-1;2} + 2 d w_{j-1;1}),
\end{equation}
with $w_{j;1} = (d+\frac12)c^{2j+1}/(1-c+c^2)^{j+1}$, and using only
the first two terms of $w_{j-1;2}$ which are known (and indicated above
for $j$). After calculation, we obtain:
\begin{equation}
w_{j;2} = \frac12(2d+1)^2 c^{2j+2}\left(1 + (j+1) c + \frac{j(j+1)}2c^2 + O(c^3)\right)\qquad(j\geq3).
\end{equation}

For $k\geq2$, we need $w_{j;2}$ with $j=k-2$. But this means
that each of $k=2$, $k=3$, $k=4$, will have its own specific behavior, and, a
priori, a common form for the coefficient of
$b^{-2k-3}$ will only apply from $k=5$ onwards. Since the exact value of this fifth
coefficient does not appear to be of great importance, and despite the space
which has been devoted here to obtaining $w_{j;2}$ with three terms, we will use this
information only for $k=2$ (therefore $w_{0;2}$). For $k\geq3$ we give
only the first four coefficients of the asymptotic expansion.

Here is the result for $k=2$:
\begin{proposition}\label{prop:develIbd2}
  Let $d>0$ and $b>d$ be two integers. The Irwin harmonic sum
  $I(b,d,2)=\sum^{(2)} m^{-1}$ where the denominators are the natural numbers
  having exactly two digits $d$ in base $b$ is such that $I(b,d,2)- b\log(b)$
  has, for $b\to\infty$ at a fixed $d$, an asymptotic expansion to all orders
  in inverse powers of $b$. The start of this expansion is:
\begin{equation*}
I(b,d,2) = b \log(b)
+ \frac{a_1(d)}{b^3}
+ \frac{a_2(d)}{b^4}
+ \frac{a_3(d)}{b^5}
+ \frac{a_4(d)}{b^6}
+ \frac{a_5(d)}{b^7} + O(b^{-8}),
\end{equation*}
with:
\begin{align*}
a_1(d) &= (d+\frac12) d^{-2}/,,
\\
a_2(d) &=-\frac{d+\frac12}{d^2(d+1)}\;,
\\
a_3(d) &= (d+\frac12)\zeta(2)
- \frac{30d^3+70d^2+47d+10}{12d^3(d+1)^2}\;,
\\ 
a_4(d) &= 
\begin{aligned}[t] 
&-(2d+1)d\zeta(3) +(2d+1)\zeta(2) 
\\ 
&-\frac{12d^6-60d^5-258d^4-342d^3-212d^2-61d-6}{12 d^4(d+1)^3}\;,
\end{aligned}
\\ 
a_5(d) &= 
\begin{aligned}[t] 
&+3(d+\frac12)d^2\zeta(4) 
-(d+\frac12)(6d+1)\zeta(3) 
\\
&-\frac{60d^7+240d^6+440d^5+561d^4+465d^3+219d^2+54d+6}{12d^5(d+1)^3}\;.
\end{aligned}
\end{align*}
\end{proposition}

Now suppose $k\geq3$. The first term of the expansion will be of order
$c^{2k-1}$. If we want a result within $O(c^{2k+3})$, therefore with four
terms, it suffices in the first sum to use $m=1$ and $m=2$ and in the second
$m=1, 2, 3$. For the first sum, we will need the $w_{j;1}$, which we know
exactly, for $j=k,k-1$, and only the leading term of $w_{k-1;2}$, which is
always $2(d+\frac12)c^{2k}$ since $k>1$. For the second sum and $m=2$, we need
the two leading terms of $w_{k-2;2}$, and therefore we must distinguish $k=3$
from $k>3$ according to Proposition \ref{prop:develu} with regard to the
second coefficient which will ultimately contribute to $c^{2k+2}$ (therefore
to $c^8$ for $k=3$). For $m=3$, we only need the leading term of $w_{k-2;3}$,
which is always $(d+\frac12)c^{2k-2}$.  So, finally, for $k=3$, there will
just be a small correction to make for the last coefficient of the asymptotic
expansion compared to the general formula valid for $k\geq3$. Here is the
final proposition of this paper.

\begin{proposition}\label{prop:develIbdk}
  Let $d>0$ and $b>d$ be two integers. Let $k\geq3$. The Irwin harmonic sum
  $I(b,d,k)=\sum^{(k)} m^{-1}$ where the denominators are the natural numbers
  having exactly $k$ times the digit $d$ in base $b$ is such that $I(b,d,k)-
  b\log(b)$ has, for $b\to\infty$ at a fixed $d$ (and $k$), an asymptotic
  expansion to all orders in inverse powers of $b$. The start of this
  expansion is:
\begin{equation*}
I(b,d,k) = b \log(b)
+ \frac{a_1(d)}{b^{2k-1}}
+ \frac{a_2(d)}{b^{2k}}
+ \frac{a_3(d)}{b^{2k+1}}
+ \frac{a_4(d)}{b^{2k+2}}
+ O(b^{-2k-3}),
\end{equation*}
with, for $k\geq4$:
\begin{align*}
a_1(d) &= (d+\frac12) d^{-2}\,,
\\
a_2(d) &=\frac{(d+\frac12)\bigl((k-2)d+k-3\bigr)}{d^2(d+1)}\;,
\\ 
a_3(d) &= 
\begin{aligned}[t] 
&(d+\frac12)\zeta(2)
\\ 
&+ (d+\frac12)\frac{(k^2-5k+4)d^3+(2k^2-12k+8)d^2+(k^2-7k+1)d-2}{2d^3(d+1)^2}\;,
\end{aligned}
\\ 
a_4(d) &=
\begin{aligned}[t] 
&(d+\frac12)\Biggl( 
-2d\zeta(3) +k\zeta(2)
\\ 
&+\frac{ 
\left(\begin{aligned}
(k^3-9k^2+14k)d^5 & +(3k^3-30k^2+45k+42)d^4
\\
& +(3k^3-33k^2+39k+114)d^3
\\
& +(k^3-12k^2+2k+108)d^2+(-6k+43)d+6
\end{aligned}\right)
}{6 d^4(d+1)^3}\Biggr).
\end{aligned}
\end{align*}
For $k=3$, the formulas are valid except for $a_4(d)$. One must then add
$(d^2-\frac13)/d^3$ to it.
\begin{equation*}
a_4(d,k=3) =
-(2d+1)d\zeta(3) +3(d+\frac12)\zeta(2)
-\frac{12d^6-50d^4-81d^3-71d^2-33d-6}{12 d^4(d+1)^3}\;.
\end{equation*}
\end{proposition}

Here are some numerical observations for $k=2$, $k=3$, $k=4$ and $k=5$. We
evaluated the values of $I(b,d,k)$ for $b=25$, $125$, $250$, $500$, $1000$, and
for $d\in\{1,\dots,10\}$ (we only show some of them of obey page constraints)
using the \textsf{SageMath} code by the author \cite{burnolirwin}, and we
multiply their deviations
from the approximations by the power of $b$ corresponding to the first omitted
term. For $k=5$, this means we multiply by $b^{13}$, so if $b=1000$, by
$10^{39}$. We performed these calculations with $68$ decimal places of precision
(before subtracting for the difference), which is excessive, but better safe
than sorry.

For this first table $k=2$, we used the five terms (from $b^{-3}$ to $b^{-7}$)
given by Proposition \ref{prop:develIbd2}. The difference is therefore
here multiplied by $b^8$.
\begin{equation*}\arraycolsep2.9pt
\begin{array}{clllln{5}{2}n{5}{2}n{5}{2}n{6}{2}}
(k=2) &\multicolumn{1}{c}{d=1}&\multicolumn{1}{c}{2}&\multicolumn{1}{c}{3}
&\multicolumn{1}{c}{4}&\multicolumn{1}{c}{5} 
&\multicolumn{1}{c}{6}&\multicolumn{1}{c}{7}&\multicolumn{1}{c}{8}
\\
b=25 &36.74&22.18&-24.62&-284.46&-969.22&-2335.30&-4675.99&-8314.67%
\\
125 &40.65&24.71&-21.32&-305.32&-1094.98&-2741.15&-5675.32&-10406.06%
\\
250 &41.19&25.07&-20.76&-307.96&-1112.67&-2801.26&-5829.43&-10740.42%
\\
500 &41.46&25.25&-20.47&-309.27&-1121.70&-2832.23&-5909.50&-10915.42%
\\
1000&41.60&25.35&-20.32&-309.93&-1126.26&-2847.95&-5950.31&-11004.97%
\end{array}
\end{equation*}

For the following three tables, we used the four terms of $b^{-2k+1}$ to
$b^{-2k-2}$ as given by Proposition \ref{prop:develIbdk}.

Difference multiplied by $b^9$:
\begin{equation*}\arraycolsep4pt
\begin{array}{cllllllln{4}{2}n{4}{2}}
(k=3) &\multicolumn{1}{c}{d=1}&\multicolumn{1}{c}{2}&\multicolumn{1}{c}{3}
&\multicolumn{1}{c}{4}&\multicolumn{1}{c}{5}
&\multicolumn{1}{c}{6}&\multicolumn{1}{c}{7}&\multicolumn{1}{c}{8}
&\multicolumn{1}{c}{9}%
\\
b=25& -9.48 & -11.82 & 6.47 & 60.13 & 160.66 & 317.67 & 539.24 & 832.18 & 1202.30%
\\
125 & -9.46 & -11.86 & 7.14 & 67.04 & 185.29 & 378.69 & 663.43 & 1055.04 & 1568.46%
\\
250 & -9.47 & -11.88 & 7.22 & 67.98 & 188.83 & 387.86 & 682.80 & 1091.05 & 1629.64%
\\
500 & -9.48&-11.88&7.25&68.45&190.64&392.59&692.89&1109.95&1661.96%
\\
1000& -9.48&-11.89&7.27&68.69&191.56&395.00&698.04&1119.62&1678.58%
\end{array}
\end{equation*}

Deviation multiplied by $b^{11}$:
\begin{equation*}\arraycolsep4pt
\begin{array}{clllllllln{4}{2}}
(k=4) &\multicolumn{1}{c}{d=1}&\multicolumn{1}{c}{2}&\multicolumn{1}{c}{3} 
&\multicolumn{1}{c}{4}&\multicolumn{1}{c}{5} 
&\multicolumn{1}{c}{6}&\multicolumn{1}{c}{7}&\multicolumn{1}{c}{8} 
&\multicolumn{1}{c}{9}%
\\
b=25& -8.00&-13.06&-2.08&40.72&127.40&267.98&470.89&743.21&1090.97%
\\
125 & -8.02&-12.68&-1.46&45.97&147.34&319.64&579.16&941.59&1422.00%
\\
250 & -8.03 & -12.64 & -1.39 & 46.67 & 150.19 & 327.36 & 596.01 & 973.59 & 1477.21%
\\
500 & -8.04 & -12.62 & -1.36 & 47.02 & 151.65 & 331.35 & 604.78 & 990.36 & 1506.37%
\\
1000 & -8.05 & -12.61 & -1.35 & 47.20 & 152.38 & 333.38 & 609.25 & 998.96 & 1521.36%
\end{array}
\end{equation*}

Difference multiplied by $b^{13}$:
\begin{equation*}\arraycolsep4pt
\begin{array}{clllllllln{4}{2}}
(k=5) &\multicolumn{1}{c}{d=1}&\multicolumn{1}{c}{2}&\multicolumn{1}{c}{3} 
&\multicolumn{1}{c}{4}&\multicolumn{1}{c}{5} 
&\multicolumn{1}{c}{6}&\multicolumn{1}{c}{7}&\multicolumn{1}{c}{8} 
&\multicolumn{1}{c}{9}%
\\
b=25& -11.06&-13.73&-7.47&26.28&100.56&225.95&411.26&663.90&990.17 %
\\
125 & -10.73 & -12.79 & -6.59 & 30.51 & 116.86 & 269.74 & 505.66 & 840.44 & 1289.31%
\\
250 & -10.70 & -12.68 & -6.49 & 31.06 & 119.16 & 276.26 & 520.30 & 868.85 & 1339.11%
\\
500 & -10.69 & -12.63 & -6.45 & 31.33 & 120.34 & 279.62 & 527.92 & 883.74 & 1365.40%
\\
1000 & -10.68&-12.60&-6.42&31.47&120.94&281.33&531.80&891.36&1378.91%
\end{array}
\end{equation*}


\clearpage
\footnotesize


\def\arxivurl#1{\href{https://arxiv.org/abs/#1}{\textsf{arXiv:#1}}}
\providecommand\bibcommenthead{}
\def\blocation#1{\unskip}


\end{document}